\documentclass{amsart}

\usepackage{amssymb, amscd, amsthm, mathrsfs}

\newtheorem{Thm}{Theorem}[section]
\newtheorem{Prop}[Thm]{Proposition}

\newtheorem{Def}[Thm]{Definition}

\theoremstyle{remark}

\newtheorem*{Notation}{Notation}
\newtheorem*{Ack}{Acknowledgement}

\numberwithin{equation}{section}

\newcommand{\Order}{\mathcal{O}}
\newcommand{\into}{\hookrightarrow}
\newcommand{\onto}{\twoheadrightarrow}
\newcommand{\isomto}{\overset{\sim}{\to}}
\newcommand{\isomfrom}{\overset{\sim}{\leftarrow}}
\newcommand{\compose}{\mathbin{\circ}}
\newcommand{\tensor}{\mathbin{\otimes}}

\newcommand{\N}{\mathbb{N}}
\newcommand{\Z}{\mathbb{Z}}

\newcommand{\et}{\mathrm{et}}
\newcommand{\fppf}{\mathrm{fppf}}

\newcommand{\id}{\mathrm{id}}
\newcommand{\Gm}{\mathbf{G}_{m}}
\newcommand{\Ga}{\mathbf{G}_{a}}

\newcommand{\dirlim}{\varinjlim}
\newcommand{\invlim}{\varprojlim}

\newcommand{\ur}{\mathrm{ur}}

\newcommand{\rat}{\mathrm{rat}}
\newcommand{\perar}{\mathrm{perar}}
\newcommand{\ind}{\mathrm{ind}}
\newcommand{\pro}{\mathrm{pro}}
\newcommand{\perf}{\mathrm{perf}}
\newcommand{\set}{\mathrm{set}}
\newcommand{\Alg}{\mathrm{Alg}}
\newcommand{\Ind}{\mathrm{I}}
\newcommand{\Pro}{\mathrm{P}}
\newcommand{\fc}{\mathrm{fc}}
\newcommand{\pre}{\mathrm{P}}
\newcommand{\alg}[1]{\mathbf{#1}}
\newcommand{\var}{\;\cdot\;}
\newcommand{\Ch}{\mathrm{Ch}}
\newcommand{\ideal}[1]{\mathfrak{#1}}
\newcommand{\ctensor}{\mathbin{\Hat{\otimes}}}
\newcommand{\assoc}{\mathsf{a}}
\newcommand{\artin}[1]{\Acute{#1}}

\newcommand{\converges}{\Longrightarrow}
\newcommand{\genby}[1]{\langle #1 \rangle}
\newcommand{\SDual}{\mathrm{SD}}
\newcommand{\proj}{\mathrm{proj}}

\DeclareMathOperator{\Hom}{Hom}

\DeclareMathOperator{\Ext}{Ext}

\DeclareMathOperator{\Spec}{Spec}
\DeclareMathOperator{\Ab}{Ab}
\DeclareMathOperator{\Set}{Set}

\let\Im\relax

\DeclareMathOperator{\Im}{Im}
\DeclareMathOperator{\sheafhom}{\alg{Hom}}
\DeclareMathOperator{\sheafext}{\alg{Ext}}

\title[An improvement of the duality formalism]
	{An improvement of the duality formalism of the rational \'etale site}
\author{Takashi Suzuki}
\address{
	Department of Mathematics, Chuo University,
	1-13-27 Kasuga, Bunkyo-ku, Tokyo 112-8551, JAPAN
}
\email{tsuzuki@gug.math.chuo-u.ac.jp}
\thanks{
	The author is a Research Fellow of Japan Society for the Promotion of Science
	and supported by JSPS KAKENHI Grant Number JP18J00415.
	This work was supported by the Research Institute for Mathematical Sciences,
	a Joint Usage/Research Center located in Kyoto University.
}
\date{October 7, 2019}

\begin{document}

\begin{abstract}
	We improve the arithmetic duality formalism of the rational \'etale site.
	This improvement allows us to avoid some exotic approximation arguments
	on local fields with ind-rational base,
	thus simplifying the proofs of the previously established duality theorems in the rational \'etale site
	and making the formalism more user-friendly.
	In a subsequent paper, this new formulation will be used in a crucial way
	to study duality for two-dimensional local rings.
\end{abstract}

\maketitle

\tableofcontents

%%%%%%%%%%%%%%%%%%%%%%%%%%%%%%%%%%%%%%%%%%%%%%%%%%%%%%%%%%%%%%%%%%%%%%%%%%%%%

\section{Introduction}
\label{sec: Introduction}

\subsection{Aim of the paper}
The arithmetic duality formalism of the rational \'etale site \cite{Suz13}
has been applied to several situations \cite{Suz14}, \cite{Suz18a}, \cite{Suz18b}, \cite{GS18}.
One of the difficulties in this formalism is that,
for a complete discrete valuation field $K$ with perfect residue field $k$ of characteristic $p > 0$,
we need to calculate the \'etale or fppf cohomology of
a certain complicated ring $\alg{K}(k')$,
where $k'$ is an arbitrary ``ind-rational $k$-algebra''.
A rational $k$-algebra is a finite product of perfections of finitely generated field extensions over $k$,
and an ind-rational $k$-algebra is a filtered direct limit of rational $k$-algebras.
The ring $\alg{K}(k')$ is the $p$-inverted ring of Witt vectors $W(k')[1 / p]$ (in the absolutely unramified case) or
the formal Laurent series ring $k'[[t]][1 / t]$.
A typical example of an ind-rational $k$-algebra is
the affine ring of a profinite set viewed as a profinite $k$-scheme.
If $k'$ has only finitely many direct factors,
then $\alg{K}(k')$ is a classical object,
since it is a finite product of complete discrete valuation fields with perfect residue fields.
Otherwise $\alg{K}(k')$ is a difficult infinite-dimensional non-noetherian ring.
We need general ind-rational algebras to describe
pro-algebraic and/or profinite group structures on cohomology of $K$,
since a profinite set tested by field-valued points is not distinguishable with a discrete set.
The \'etale cohomology $H^{n}(\alg{K}(k'), \Gm)$ and
more general $H^{n}(\alg{K}(k'), G)$ for smooth group schemes $G$
(in particular, abelian varieties) over $K$
are calculated based on some exotic approximation arguments
in \cite[Section 2.5]{Suz13} and \cite[Sections 3.1 and 3.2]{Suz14}, respectively.

In this paper, we give a simpler and more user-friendly formalism
that does not require exotic approximation arguments.
In this new formalism, we only need to calculate $H^{n}(\alg{K}(k'), G)$
for \emph{perfect field extensions} $k'$ over $k$,
in which case $\alg{K}(k')$ is a genuinely classical object as explained above.
The key observation is that for most of the groups of interest $G$,
the $\pi_{0}$ (component group) of the object representing the sheafification of the presheaf
$k' \mapsto H^{n}(\alg{K}(k'), G)$ turns out to be an \'etale $k$-group
(without a profinite part).
Pro-algebraic groups with finite (that is, not profinite) component groups
can be described by perfect-field-valued points alone,
as we will see in this paper.
Hence we may restrict $k'$ to be perfect fields.
We still need arbitrary perfect fields here and not only perfections of finitely generated fields
or rational $k$-algebras,
since the generic point of a connected pro-algebraic group is not
the spectrum of the perfection of a finitely generated field.
Once we uniquely pin down such pro-algebraic groups by perfect-field-valued points,
we can then pass to the pro-\'etale site of ind-rational $k$-algebras,
where we have full control of the derived categories of pro-algebraic groups and of profinite groups.

This new formalism will be useful and in fact necessary
for \emph{two-dimensional local rings} such as $W(k)[[t]]$,
since the sheafification of the presheaf
	\[
			k'
		\mapsto
			H^{n} \bigl(
				W(k')[[t]][1 / p], \Z / p \Z(r)
			\bigr)
	\]
on ind-rational $k$-algebras $k'$ does not commute with filtered direct limits
(since the representing object should be a pro-algebraic group)
and hence an analogue of the approximation arguments mentioned above are
not just difficult but in fact impossible (at least when interpreted naively).
In a subsequent paper, using the explicit computations of filtrations by symbols
in the proof of \cite[Claim (4.11)]{Sai86},
the above sheaf will be shown to be representable
by a pro-algebraic group over $k$ with finite $\pi_{0}$
if $k'$ runs over perfect field extensions of $k$.
The purpose of the proposed paper will be to use the formalism of this paper
to construct a duality theory for such pro-algebraic groups associated with
two-dimensional noetherian complete normal local rings
of mixed characteristic with perfect residue field,
extending Saito's duality theories \cite{Sai86}, \cite{Sai87} in the finite residue field case.

In this paper,
emphasis is put on providing a dictionary between the older and new formalisms,
so that the reader can freely translate the duality results
previously established in the older formalism into the new formalism
and use them in the new formalism.
We will also provide enough foundational results on the new formalism
so that it can be used on its own
(without translating back into the older formalism)
to explore new duality results in future work.

%%%%%%%%%%%%%%%%%%%%%%%%%%%%%%%%%%%%%%%%%%%%%%%%%%%%%%%%%%%%%%%%%%%%%%%%%%%%%

\subsection{Main theorems}
Now we formulate our results.
Let $k$ be a perfect field of characteristic $p > 0$.
Let $k^{\perar}$ be the category of finite products of perfect field extensions of $k$
with $k$-algebra homomorphisms
(where ``$\perar$'' stands for perfect artinian).
Define $\Spec k^{\perar}_{\et}$ to be the \'etale site on the category $k^{\perar}$,
which we call the \emph{perfect artinian \'etale site} of $k$.

Also let $k^{\perf'}$ be the category of quasi-compact quasi-separated perfect $k$-schemes.
This category can be equipped with the ``pro-fppf'' topology
(\cite[Remark 3.8.4]{Suz13}, \cite[Appendix A]{Suz14}).
Denote the resulting site by $\Spec k^{\perf'}_{\pro\fppf}$.
The inclusion functor $k^{\perar} \into k^{\perf'}$ induces a morphism of topologies
(or a ``premorphism of sites'' \cite[Section 2.4]{Suz18a})
	\[
			\artin{h}
		\colon
			\Spec k^{\perf'}_{\pro\fppf}
		\to
			\Spec k^{\perar}_{\et}.
	\]
Its pullback functor $\artin{h}^{\ast} \colon \Ab(k^{\perar}_{\et}) \to \Ab(k^{\perf'}_{\pro\fppf})$
on the category of sheaves of abelian groups on these sites
admits a left derived functor $L \artin{h}^{\ast}$
by \cite[Lemma 3.7.2 and Section 2.1]{Suz13}.
Let $\Alg / k$ be the category of perfections (inverse limit along Frobenius morphisms)
of commutative algebraic groups over $k$.
Let $\Pro'_{\fc} \Alg / k$ be the full subcategory of the pro-category of $\Alg / k$
of pro-objects with affine transition morphisms and finite \'etale $\pi_{0}$
(where ``$\fc$'' stands for ``finite component (group)'').
It is a full subcategory of $\Ab(k^{\perf'}_{\pro\fppf})$ via the Yoneda functor.

\begin{Thm}[$=$ Theorem \ref{thm: derived pull of proalg group with finite components},
	Proposition \ref{prop: EIP embeds into sheaves on perar site}]
	\label{thm: main, derived pull}
	The Yoneda functor $\Pro'_{\fc} \Alg / k \to \Ab(k^{\perar}_{\et})$ is fully faithful.
	For any $G \in \Pro'_{\fc} \Alg / k$,
	the natural morphism $\artin{h}^{\ast} G \to G$ in $\Ab(k^{\perf'}_{\pro\fppf})$ is an isomorphism
	and $L_{n} \artin{h}^{\ast} G = 0$ for $n \ge 1$.
\end{Thm}

This means that treating $G$ as a functor on perfect field extensions of $k$
does not lose any information, higher derived or not.
This is a version of \cite[Proposition 3.7.3]{Suz13} for $\Spec k^{\perar}_{\et}$.
Similar to  \cite[Section 3]{Suz13},
the key points of the proof are that
the inclusion morphism $\xi_{G} \into G$
(which is not of finite presentation)
of the generic point $\xi_{G}$ of a group $G \in \Pro'_{\fc} \Alg / k$
may appear in a covering family for the site $\Spec k^{\perf'}_{\pro\fppf}$,
the restriction $\xi_{G} \times_{k} \xi_{G} \to G$ of the group operation map
(which is not pro-\'etale) is a covering for the site $\Spec k^{\perf'}_{\pro\fppf}$,
and that $\xi_{G} \in k^{\perar}$.
In Sections \ref{sec: Review of Mac Lane resolution}--\ref{sec: Acyclicity of the pullback of Mac Lane resolution},
the proof of the above theorem will be given by checking that
arguments in \cite[Sections 3.5 and 3.6]{Suz13} on Mac Lane's resolutions
may be carefully modified to work in the present setting.

Using this theorem, we can translate the duality results of
\cite{Suz13}, \cite{Suz14}, \cite{Suz18a} and \cite{Suz18b} in this setting.
We take \cite{Suz14} as an example to explain this translation.
Let $K$ be a complete discrete valuation field with ring of integer $\Order_{K}$
whose residue field is the above $k$.
For $k' \in k^{\perar}$, we define a $K$-algebra by
	\[
			\alg{K}(k')
		=
			(W(k') \Hat{\tensor}_{W(k)} \Order_{K}) \tensor_{\Order_{K}} K
	\]
(see Section \ref{sec: A duality formalism for local fields} for more detail),
which is a finite product of complete discrete valuation fields with perfect residue fields.
This functor $\alg{K}$ defines a premorphism of sites
$\artin{\pi}_{K} \colon \Spec K_{\fppf} \to \Spec k^{\perar}_{\et}$.
Let $\Spec k^{\ind\rat}_{\pro\et}$ be the ind-rational pro-\'etale site of $k$
(\cite[Section 2.1]{Suz14}).
Let $\Tilde{h} \colon \Spec k^{\perf'}_{\pro\fppf} \to \Spec k^{\ind\rat}_{\pro\et}$
be the premorphism of sites defined by the inclusion functor on the underlying categories.
For $G \in D(K_{\fppf})$, define
	\[
			R \artin{\alg{\Gamma}}(K, G)
		=
			R \Tilde{h}_{\ast} L \artin{h}^{\ast} R (\artin{\pi}_{K})_{\ast} G
		\in
			D(k^{\ind\rat}_{\pro\et}).
	\]
For most of the groups of interest $G$,
the object $R (\artin{\pi}_{K})_{\ast} G$ is ``$\artin{h}$-acyclic'',
which implies the existence of a spectral sequence
	\[
			E_{2}^{i j}
		=
			H^{i} \bigl(
				k'_{\pro\et}, \artin{\alg{H}}^{j}(K, G)
			\bigr)
		\converges
			H^{i + j}(\alg{K}(k')_{\fppf}, G)
	\]
for any $k' \in k^{\perar}$ (where $\artin{\alg{H}}^{n} = H^{n} R \artin{\alg{\Gamma}}$)
and an isomorphism
	\[
			\artin{\alg{H}}^{n}(K, G)(k')
		\cong
			H^{n}(\alg{K}(k')_{\fppf}, G)
	\]
for algebraically closed field extensions $k'$ of $k$ and any $n$.
Applying the above theorem for $G = R^{n} (\artin{\pi}_{K})_{\ast} \Gm$
and $G = R^{n} (\artin{\pi}_{K})_{\ast} A$ for an abelian variety $A / K$,
and comparing with the duality result \cite[Theorem 4.1.2]{Suz14} in the older formulation,
we obtain the following.

\begin{Thm}[$=$ Theorem \ref{thm: local duality}]
	Let $A$ and $B$ be abelian varieties over $K$ dual to each other.
	Then there exists a canonical isomorphism
		\[
				R \artin{\alg{\Gamma}}(K, A)^{\SDual \SDual}
			\isomto
				R \artin{\alg{\Gamma}}(K, B)^{\SDual}[1]
		\]
	in $D(k^{\ind\rat}_{\pro\et})$,
	where $\SDual$ denotes the derived sheaf-Hom $R \sheafhom_{k^{\ind\rat}_{\pro\et}}(\var, \Z)$
	for $\Spec k^{\ind\rat}_{\pro\et}$.
\end{Thm}

Before stating this theorem, in Section \ref{sec: Trace morphisms and a finiteness property of cohomology},
we will see that some part of its proof
is much easier to prove in this new formulation.
The purpose of the mentioned section is to clearly present how to practically use the new formulation.
We will give a direct proof of the existence of the trace morphism
$R (\artin{\pi}_{K})_{\ast} \Gm \to \Z$
and the fact that $R^{n} (\artin{\pi}_{K})_{\ast} G$ as a functor $k^{\perar} \to \Ab$
commutes with filtered direct limits that exist in $k^{\perar}$
if $G$ is a smooth group scheme over $K$ and $n \ge 1$.
The proofs \cite[Section 2.5]{Suz13} and \cite[Sections 3.1 and 3.2]{Suz14} of
the corresponding statements in the older formulation
are some exotic approximation arguments.
The direct proofs we give here are based on much more standard facts
on complete discrete valuation fields.

A remark is that it seems possible to completely eliminate ind-rational $k$-algebras
from the formulation.
The target site of
$\Tilde{h} \colon \Spec k^{\perf'}_{\pro\fppf} \to \Spec k^{\ind\rat}_{\pro\et}$
may be likely replaced by the category of
filtered inverse limits of perfections of quasi-compact smooth $k$-schemes
with affine transition morphisms
endowed with the pro-\'etale topology.
But this change would require us to redo large part of \cite{Suz13} and \cite{Suz14} with this new site
and thus take many pages.
We will not try doing this here.

As above, the notation is necessarily complicated in order to ensure compatibility
and provide a dictionary between the older and new formalisms.
It is hoped to completely renew the notation, abandon everything old and write down proofs of the results
entirely in the new formalism some time in the future.
Meanwhile, we will explain the notation in this paper as much as possible
to remedy the notational difficulties.

\begin{Ack}
	The author is grateful to Kazuya Kato for his encouragement
	to improve the formulation
	towards duality for two-dimensional local rings,
	and to the referee for their very thorough comments
	to make the paper more readable.
\end{Ack}

\begin{Notation}
	(This part is partially taken from \cite[Section 1.3, Notation]{Suz18b}.)
	The categories of sets and abelian groups are denoted by
	$\Set$ and $\Ab$, respectively.
	We denote the ind-category of a category $\mathcal{C}$ by $\Ind \mathcal{C}$,
	the pro-category by $\Pro \mathcal{C}$,
	so that $\Ind \Pro \mathcal{C} := \Ind(\Pro \mathcal{C})$ is the ind-category of $\Pro \mathcal{C}$.
	All groups, group schemes and sheaves of groups are assumed commutative.
	For an abelian category $\mathcal{A}$,
	the category of complexes in $\mathcal{A}$ in cohomological grading is denoted by $\Ch(\mathcal{A})$.
	If $A \to B$ is a morphism in $\Ch(\mathcal{A})$,
	then its mapping cone is denoted by $[A \to B]$.
	The homotopy category of $\Ch(\mathcal{A})$ is denoted by $K(\mathcal{A})$
	with derived category $D(\mathcal{A})$.
	If we say $A \to B \to C$ is a distinguished triangle in a triangulated category,
	we implicitly assume that a morphism $C \to A[1]$ to the shift of $A$ is given,
	and the triangle $A \to B \to C \to A[1]$ is distinguished.
	For a triangulated category $\mathcal{D}$ and a collection of objects $\mathcal{I}$,
	we denote by $\genby{\mathcal{I}}_{\mathcal{D}}$
	the smallest triangulated full subcategory of $\mathcal{D}$
	closed under isomorphism.
	For a site $S$, the categories of sheaves of sets and abelian groups are denoted by
	$\Set(S)$ and $\Ab(S)$, respectively.
	We denote $\Ch(S) = \Ch(\Ab(S))$
	and use the notation $K(S)$, $D(S)$ similarly.
	The Hom and sheaf-Hom functors for $\Ab(S)$ are denoted by
	$\Hom_{S}$ and $\sheafhom_{S}$, respectively.
	Their right derived functors are denoted by
	$\Ext_{S}^{n}$, $R \Hom_{S}$ and $\sheafext_{S}^{n}$, $R \sheafhom_{S}$, respectively.
	The tensor product functor $\tensor$ is over the ring $\Z$
	(or, on some site, the sheaf of rings $\Z$).
	Its left derived functor is denoted by $\tensor^{L}$.
	
	Here is the list of sites and (pre)morphisms to be defined in this paper:
		\[
				\Spec k^{\perf'}_{\pro\fppf}
			\overset{\Tilde{h}}{\to}
				\Spec k^{\ind\rat}_{\pro\et}
			\overset{\varepsilon}{\to}
				\Spec k^{\ind\rat}_{\et}
			\overset{\alpha}{\to}
				\Spec k^{\perar}_{\et}
		\]
	whose composite is $\artin{h}$;
		\[
			\begin{CD}
					\Spec K_{\fppf}
				@>> j >
					\Spec \Order_{K, \fppf}
				@>> \pi_{\Order_{K}} >
					\Spec k^{\ind\rat}_{\et}
				\\
				@|
				@|
				@V \alpha VV
				\\
					\Spec K_{\fppf}
				@> j >>
					\Spec \Order_{K, \fppf}
				@> \artin{\pi}_{\Order_{K}} >>
					\Spec k^{\perar}_{\et},
			\end{CD}
		\]
	where the composite of the upper (resp.\ lower) horizontal two morphisms
	is $\pi_{K}$ (resp.\ $\artin{\pi}_{K}$);
	and $\assoc := \varepsilon^{\ast}$, $\artin{\assoc} := R \Tilde{h}_{\ast} L \artin{h}^{\ast}$.
\end{Notation}

%%%%%%%%%%%%%%%%%%%%%%%%%%%%%%%%%%%%%%%%%%%%%%%%%%%%%%%%%%%%%%%%%%%%%%%%%%%%%

\section{Generalities on Grothendieck sites}

We mostly follow the terminology of \cite{AGV72a} on Grothendieck sites.
See also \cite{Art62} and \cite{KS06}.
We do use the modified terminology given in \cite[Section 2.4]{Suz18a};
see there for more details.
We need three classes of maps between sites:
morphisms of sites, premorphisms of sites and continuous maps of sites.
This list is roughly in decreasing order of strength.
It is not exactly so since
the notion of premorphism of sites is meaningful only for sites defined by pretopologies (or covering families)
and it depends on the choice of the pretopologies.
It is this intermediate notion that we encounter most in practice in this paper.

First we recall the weakest notion, continuous maps of sites, and related notions.

\begin{Def} \mbox{}
	\begin{enumerate}
		\item
			For sites $S$ and $S'$,
			a \emph{continuous map of sites} $f \colon S' \to S$
			(called a continuous functor from $S$ to $S'$ in \cite[Expos\'e III, D\'efinition 1.1]{AGV72a})
			is a functor $f^{-1}$ from the underlying category of $S$ to that of $S'$
			such that the right composition with $f^{-1}$ (or the pushforward functor $f_{\ast}$)
			sends sheaves of sets on $S'$ to sheaves of sets on $S$.
		\item
			In this case, $f_{\ast} \colon \Set(S') \to \Set(S)$ and $f_{\ast} \colon \Ab(S') \to \Ab(S)$ have
			left adjoints (the \emph{pullback functors}),
			which we denote by $f^{\ast \set} \colon \Set(S) \to \Set(S')$ and
			$f^{\ast} \colon \Ab(S) \to \Ab(S')$, respectively.
		\item
			If we write $X \in S$, we mean that $X$ is an object of the underlying category of $S$.
		\item
			For $X \in S$,
			the \emph{localization} (\cite[Expos\'e II, 5.1]{AGV72a}) of $S$ at $X$ is denoted by $S / X$.
		\item
			The \emph{restriction} (\cite[Expos\'e II, 5.3, 2)]{AGV72a}) of $F \in \Set(S)$
			(or $\in \Ab(S)$ or $\in D(S)$) to $S / X$ is denoted by $F|_{X}$.
		\item
			We denote by $f_{X} \colon S' / f^{-1} X \to S / X$
			the continuous map of sites defined by the restriction of $f^{-1}$ on the localizations.
	\end{enumerate}
\end{Def}

Next we recall morphisms of sites.

\begin{Def}
	Let $f \colon S' \to S$ be a continuous map of sites.
	If $f^{\ast \set}$ is exact
	(i.e.\ commutes with finite inverse limits),
	we say that $f$ is a \emph{morphism of sites}.
\end{Def}

In this case, $f^{\ast}$ and $f^{\ast \set}$ are compatible with
forgetting group structures (\cite[III, Proposition 1.7,4]{AGV72a}),
so we do not have to distinguish them.

The exactness of $f^{\ast\set}$ is usually too much to ask
if the underlying category of $S$ does not have all finite inverse limits.
But it is inconvenient if we make no assumption on exactness of $f^{\ast\set}$.
Some exactness on at least representable presheaves helps much.
In this regard, the following notion, premorphisms of sites, is useful,
which we recall from \cite[Section 2.4]{Suz18a}.

\begin{Def}
	Let $S$ and $S'$ be sites defined by pretopologies.
	A \emph{premorphism of sites} $f \colon S' \to S$
	is a functor $f^{-1}$ from the underlying category of $S$ to the underlying category of $S'$
	that sends covering families to covering families
	such that $f^{-1}(Y \times_{X} Z) \isomto f^{-1} Y \times_{f^{-1} X} f^{-1} Z$
	whenever $Y \to X$ appears in a covering family.
\end{Def}

Such a functor $f^{-1}$ is called a morphism of topologies from $S$ to $S'$
in \cite[Definition 2.4.2]{Art62}.
In this case, $f$ defines a continuous map of sites $f \colon S' \to S$,
and by \cite[Lemma 3.7.2]{Suz13} and the first paragraph of \cite[Section 2.1]{Suz13},
the functor $f^{\ast} \colon \Ab(S) \to \Ab(S')$ admits a left derived functor
$L f^{\ast} \colon D(S) \to D(S')$, which is left adjoint to $R f_{\ast} \colon D(S') \to D(S)$.
Be careful that the coefficient ring for sheaves here is $\Z$,
and there is nothing analogous to the functors $L_{n} f^{\ast}$ for $n \ge 1$
if one considers only morphisms of sites.
They are not analogous to $\operatorname{Tor}_{n}^{R}(S, \var)$ for a ring homomorphism $R \to S$
or $L_{n} g^{\ast}$ for a scheme morphism $g$ and coherent sheaves.

Now let $f \colon S' \to S$ be a continuous map of sites
with underlying functor $f^{-1}$ on the underlying categories.
We need a cup product morphism relative to $f$
(assuming nothing about exactness of $f^{\ast}$).
The following was essentially observed in \cite[(2.5.2)]{Suz18a} in a special case.

\begin{Prop} \label{prop: sheafified functoriality and cup product}
	There exist canonical morphisms
		\begin{gather}
				\label{eq: derived push and RHom}
					R f_{\ast}
					R \sheafhom_{S'}(G', F')
				\to
					R \sheafhom_{S}(R f_{\ast} G', R f_{\ast} F'),
			\\
				\label{eq: cup product}
					R f_{\ast} G' \tensor^{L} R f_{\ast} F'
				\to
					R f_{\ast}(G' \tensor^{L} F')
		\end{gather}
	in $D(S)$ functorial in $G', F' \in D(S')$.
\end{Prop}

Note that this type of statements is usually proved
under the assumption that $f$ is a morphism of sites
and making use of this assumption.

\begin{proof}
	We construct \eqref{eq: derived push and RHom}.
	First, let $G', F' \in \Ab(S')$.
	The functoriality of $f_{\ast}$ gives a canonical homomorphism
		\begin{equation} \label{eq: functoriality of push in abelian category setting}
				\Hom_{S' / f^{-1} X}(G'|_{f^{-1} X}, F'|_{f^{-1} X})
			\to
				\Hom_{S / X}((f_{\ast} G')|_{X}, (f_{\ast} F')|_{X})
		\end{equation}
	functorial in $X \in S$.
	Hence we have a morphism
		\[
				f_{\ast} \sheafhom_{S'}(G', F')
			\to
				\sheafhom_{S}(f_{\ast} G', f_{\ast} F')
		\]
	in $\Ab(S)$ functorial in $G', F' \in \Ab(S')$.
	This extends to a morphism in the category of complexes $\Ch(S)$
	functorial in $G', F' \in \Ch(S')$,
	where $\sheafhom$ is understood to be the total complex of the sheaf-Hom double complex.
	This further extends to a morphism in the homotopy category $K(S)$
	functorial in $G', F' \in K(S')$.
	Composing with the localization $\sheafhom_{S} \to R \sheafhom_{S}$ on the right-hand side,
	we have a morphism
		\[
				f_{\ast} \sheafhom_{S'}(G', F')
			\to
				R \sheafhom_{S}(f_{\ast} G', f_{\ast} F')
		\]
	in $D(S)$ functorial in $G', F' \in K(S')$.
	If $F'$ is K-injective (or homotopically injective \cite[Definition 14.1.4 (i)]{KS06}),
	then $\sheafhom_{S'}(G', F')$ is K-limp (\cite[Section 2.4, Proposition 2.4.1]{Suz18a})
	and hence $f_{\ast}$-injective by \cite[Proposition 2.4.2]{Suz18a}.
	Hence the left-hand side $f_{\ast} \sheafhom_{S'}(G', F')$ represents
	$R f_{\ast} R \sheafhom_{S'}(G', F')$.
	If moreover $G'$ is K-injective, then the right-hand side is isomorphic to
	$R \sheafhom_{S}(R f_{\ast} G', R f_{\ast} F')$.
	Hence we have a morphism
		\[
				R f_{\ast} R \sheafhom_{S'}(G', F')
			\to
				R \sheafhom_{S}(R f_{\ast} G', R f_{\ast} F')
		\]
	in $D(S)$ functorial in the objects $G', F'$
	of the homotopy category of K-injective complexes in $\Ab(S')$.
	Since the homotopy category of K-injective complexes in $\Ab(S')$ is equivalent to $D(S')$
	(\cite[Corollary 14.1.12 (i)]{KS06}),
	we have the morphism \eqref{eq: derived push and RHom}.
	
	We construct \eqref{eq: cup product}.
	The morphism \eqref{eq: derived push and RHom} gives a morphism
		\[
				R f_{\ast}
				R \sheafhom_{S'}(G', G' \tensor^{L} F')
			\to
				R \sheafhom_{S} \bigl(
					R f_{\ast} G',
					R f_{\ast}(G' \tensor^{L} F')
				\bigr).
		\]
	By the derived tensor-Hom adjunction (\cite[Theorem 18.6.4 (vii)]{KS06}),
	we have a morphism
		\[
				R f_{\ast} G' \tensor^{L} R f_{\ast} R \sheafhom_{S'}(G', G' \tensor^{L} F')
			\to
				R f_{\ast }(G' \tensor^{L} F').
		\]
	By composing it with the evaluation morphism
		\[
			F' \to R \sheafhom_{S'}(G', G' \tensor^{L} F'),
		\]
	we obtain the morphism \eqref{eq: cup product}.
\end{proof}

As one can see from the above proof,
the key point was the part that shows $f_{\ast} \sheafhom_{S'}(G', F')$ represents
$R f_{\ast} R \sheafhom_{S'}(G', F')$ if $F'$ is K-injective.

Next assume that $S$ and $S'$ are sites defined by pretopologies
and $f \colon S' \to S$ is a premorphism of sites.
The derived pullback $L f^{\ast}$ is difficult to handle in general.
There are two senses in which $L f^{\ast}$ is controllable:

\begin{Def} \label{def: f compatible, f acyclic} \mbox{}
	\begin{enumerate}
		\item
			We say that an object $F \in D(S)$ is \emph{$f$-compatible}
			if the natural morphism $L(f|_{X})^{\ast}(F|_{X}) \to (L f^{\ast} F)|_{f^{-1} X}$
			is an isomorphism for any $X \in S$.
		\item
			We say that $F$ is \emph{(weakly) $f$-acyclic}
			if the natural morphism $F \to R f_{\ast} L f^{\ast} F$ is an isomorphism.
	\end{enumerate}
\end{Def}

The $f$-compatibility is automatically true (for any $F$) if $f$ is a morphism of sites
(essentially stated in \cite[IV (5.10.1)]{AGV72a}).
It can fail in general; see \cite[Remark 3.5.2]{Suz13}.
Also see \cite[Proposition 3.1 (1)]{Suz18b} for a certain positive result.
On the other hand, the similar morphism
$(R f_{\ast} F')|_{X} \to R(f|_{X})_{\ast}(F'|_{f^{-1} X})$
is always an isomorphism for any $F' \in D(S')$.
What is weak in the definition of $f$-acyclicity is that
we do not require each cohomology object of $F$ to satisfy the same condition.
If $F \to F' \to F''$ is a distinguished triangle in $D(S)$
and if $F$ and $F'$ are $f$-compatible (resp.\ $f$-acyclic),
then so is $F''$.

\begin{Prop} \label{prop: composite of derived pullback}
	If $S''$ is another site defined by a pretopology
	and $g \colon S'' \to S'$ a premorphism of sites.
	Then $R (f \compose g)_{\ast} \isomto R f_{\ast} \compose R g_{\ast}$ as $D(S'') \to D(S)$
	and $L g^{\ast} \compose L f^{\ast} \isomto L (f \compose g)^{\ast}$ as $D(S) \to D(S'')$.
\end{Prop}

\begin{proof}
	The statement about the pushforward follows from \cite[Propositions 2.4.2 and 2.4.3]{Suz18a}.
	This implies the other statement by adjunction.
\end{proof}

In the next two propositions,
we relate the morphism \eqref{eq: derived push and RHom} to $L f^{\ast}$.

\begin{Prop}
	The morphism \eqref{eq: derived push and RHom}
	after applying $R \Gamma(X, \var)$ for any $X \in S$
	can be canonically identified with the composite
		\begin{align*}
			&
					R \Hom_{S' / f^{-1} X}(G'|_{f^{-1} X}, F'|_{f^{-1} X})
			\\
			&	\to
					R \Hom_{S' / f^{-1} X} \bigl(
						L(f|_{X})^{\ast} R(f|_{X})_{\ast}(G'|_{f^{-1} X}), F'|_{f^{-1} X}
					\bigr)
			\\
			&	\cong
					R \Hom_{S / X} \bigl(
						R(f|_{X})_{\ast}(G'|_{f^{-1} X}), R(f|_{X})_{\ast}(F'|_{f^{-1} X})
					\bigr)
		\end{align*}
	in $D(\Ab)$,
	where the first morphism is induced by the counit of adjunction
	and the second isomorphism is the adjunction.
\end{Prop}

\begin{proof}
	For K-injective $G', F' \in K(S')$,
	the composite morphism in the statement can be written as
		\begin{align*}
			&
					\Hom_{S' / f^{-1} X}(G'|_{f^{-1} X}, F'|_{f^{-1} X})
			\\
			&	\to
					\Hom_{S' / f^{-1} X} \bigl(
						(f|_{X})^{\ast} (f|_{X})_{\ast}(G'|_{f^{-1} X}), F'|_{f^{-1} X}
					\bigr)
			\\
			&	\to
					R \Hom_{S' / f^{-1} X} \bigl(
						L(f|_{X})^{\ast} (f|_{X})_{\ast}(G'|_{f^{-1} X}), F'|_{f^{-1} X}
					\bigr)
			\\
			&	\cong
					R \Hom_{S / X} \bigl(
						(f|_{X})_{\ast}(G'|_{f^{-1} X}), (f|_{X})_{\ast}(F'|_{f^{-1} X})
					\bigr)
		\end{align*}
	in $D(\Ab)$,
	where $\Hom$ is understood to be the total complex of the $\Hom$ double complex.
	Hence it can also be written as
		\begin{align*}
			&
					\Hom_{S' / f^{-1} X}(G'|_{f^{-1} X}, F'|_{f^{-1} X})
			\\
			&	\to
					\Hom_{S' / f^{-1} X} \bigl(
						(f|_{X})^{\ast} (f|_{X})_{\ast}(G'|_{f^{-1} X}), F'|_{f^{-1} X}
					\bigr)
			\\
			&	\cong
					\Hom_{S / X} \bigl(
						(f|_{X})_{\ast}(G'|_{f^{-1} X}), (f|_{X})_{\ast}(F'|_{f^{-1} X})
					\bigr)
			\\
			&	\to
					R \Hom_{S / X} \bigl(
						(f|_{X})_{\ast}(G'|_{f^{-1} X}), (f|_{X})_{\ast}(F'|_{f^{-1} X})
					\bigr).
		\end{align*}
	The composite morphism from the first term to the third term can be identified with
	the morphism \eqref{eq: functoriality of push in abelian category setting}.
	This implies the result.
\end{proof}

\begin{Prop}
	For $G \in D(S)$ and $F' \in D(S')$, consider the composite
		\[
				R f_{\ast}
				R \sheafhom_{S'}(L f^{\ast} G, F')
			\to
				R \sheafhom_{S}(R f_{\ast} L f^{\ast} G, R f_{\ast} F')
			\to
				R \sheafhom_{S}(G, R f_{\ast} F')
		\]
	of the morphism \eqref{eq: derived push and RHom}
	and the unit of adjunction.
	This is an isomorphism if $G$ is $f$-compatible.
\end{Prop}

\begin{proof}
	It is enough to show that the stated morphism becomes an isomorphism in $D(\Ab)$
	when $R \Gamma(X, \var)$ is applied for any $X \in S$.
	Let $\varphi \colon L(f|_{X})^{\ast}(G|_{X}) \to (L f^{\ast} G)|_{f^{-1} X}$ be
	the natural morphism, which is an isomorphism under the assumption.
	By the previous proposition, the morphism after $R \Gamma(X, \var)$ is given by the composite
		\begin{align*}
			&
					R \Hom_{S' / f^{-1} X} \bigl(
						(L f^{\ast} G)|_{f^{-1} X},
						F'|_{f^{-1} X}
					\bigr)
			\\
			&	\to
					R \Hom_{S' / f^{-1} X} \Bigl(
						L (f|_{X})^{\ast}
						R (f|_{X})_{\ast}
						\bigl(
							(L f^{\ast} G)|_{f^{-1} X}
						\bigr),
						F'|_{f^{-1} X}
					\Bigr)
			\\
			&	\cong
					R \Hom_{S / X} \bigl(
						(R f_{\ast} L f^{\ast} G)|_{X},
						(R f_{\ast} F')|_{X}
					\bigr)
			\\
			&	\to
					R \Hom_{S / X} \bigl(
						G|_{X},
						(R f_{\ast} F')|_{X}
					\bigr)
		\end{align*}
	of the counit of adjunction, the adjunction isomorphism
	and the unit of adjunction.
	The morphism $G|_{X} \to (R f_{\ast} L f^{\ast} G)|_{X}$ used in the third morphism
	can be written as the composite
		\[
				G|_{X}
			\to
				R (f|_{X})_{\ast} L(f|_{X})^{\ast}(G|_{X})
			\to
				R (f|_{X})_{\ast} \bigl(
					(L f^{\ast} G)|_{f^{-1} X}
				\bigr)
		\]
	of the unit of adjunction and the morphism $\varphi$.
	Hence the morphism after $R \Gamma(X, \var)$ can also be written as the composite
		\begin{align*}
			&		R \Hom_{S' / f^{-1} X} \bigl(
						(L f^{\ast} G)|_{f^{-1} X},
						F'|_{f^{-1} X}
					\bigr)
			\\
			&	\to
					R \Hom_{S' / f^{-1} X} \bigl(
						L(f|_{X})^{\ast}(G|_{X}),
						F'|_{f^{-1} X}
					\bigr)
			\\
			&	\cong
					R \Hom_{S / X} \bigl(
						G|_{X},
						(R f_{\ast} F')|_{X}
					\bigr)
		\end{align*}
	of $\varphi$ and the adjunction isomorphism
	since $\varphi$ and the adjunction isomorphism commute with the unit and the counit
	and the composite of the counit and the unit is the identity.
	This composite is an isomorphism if $\varphi$ is an isomorphism.
	Hence the result follows.
\end{proof}

Using the above, we obtain a compatibility between $L f^{\ast}$ and $\tensor^{L}$
under an $f$-compatibility assumption:

\begin{Prop} \label{prop: cocup product}
	For any $G, F \in D(S)$,
	consider the morphism
		\[
				L f^{\ast}(G \tensor^{L} F)
			\to
				L f^{\ast} G \tensor^{L} L f^{\ast} F
		\]
	corresponding to the composite
		\[
				G \tensor^{L} F
			\to
				R f_{\ast} L f^{\ast} G \tensor^{L} R f_{\ast} L f^{\ast} F
			\to
				R f_{\ast}(L f^{\ast} G \tensor^{L} L f^{\ast} F)
		\]
	of the unit of adjunction and
	the morphism \eqref{eq: cup product}.
	This morphism is an isomorphism if $G$ or $F$ is $f$-compatible.
\end{Prop}

\begin{proof}
	We may assume that $F$ is $f$-compatible.
	For any $F' \in D(S')$, we have isomorphisms
		\begin{align*}
			&		R \Hom_{S'}(
						L f^{\ast} G \tensor^{L} L f^{\ast} F,
						F'
					)
			\\
			&	\cong
					R \Hom_{S} \bigl(
						G,
						R f_{\ast}
						R \sheafhom_{S'}(
							L f^{\ast} F, F'
						)
					\bigr)
			\\
			&	\cong
					R \Hom_{S} \bigl(
						G,
						R \sheafhom_{S}(
							F, R f_{\ast} F'
						)
					\bigr)
			\\
			&	\cong
					R \Hom_{S'} \bigl(
						L f^{\ast}(G \tensor^{L} F),
						F'
					\bigr)
		\end{align*}
	in $D(\Ab)$ functorial in $F'$,
	where the second isomorphism is given by the previous proposition.
	This implies the result.
\end{proof}

%%%%%%%%%%%%%%%%%%%%%%%%%%%%%%%%%%%%%%%%%%%%%%%%%%%%%%%%%%%%%%%%%%%%%%%%%%%%%

\section{Derived pullback to the perfect artinian \'etale site}

In the rest of the paper,
we let $k$ be a perfect field of characteristic $p > 0$.
We fix our basic terminology:

\begin{Def} \mbox{}
	\begin{enumerate}
		\item
			A \emph{perfect field extension} of $k$ is
			a field extension of $k$ that is a perfect field.
		\item
			A $k$-scheme is said to be \emph{perfect}
			if its (relative or absolute) Frobenius morphism is an isomorphism.
		\item
			For a $k$-algebra (resp.\ a $k$-scheme),
			its \emph{perfection} is the direct (resp.\ inverse) limit along Frobenius morphisms on it.
	\end{enumerate}
\end{Def}

See \cite{BGA18} for a general reference on perfect schemes.
In \cite[Section 5]{BGA18}, the perfection of a $k$-scheme is called the inverse perfection.
A perfect field extension does not have to be
the perfection of a finitely generated extension of $k$.

\begin{Def} \mbox{}
	\begin{enumerate}
		\item
			Define $k^{\perar}$ to be the category of finite products of perfect field extensions of $k$
			with $k$-algebra homomorphisms.
		\item
			For any $k' = \prod k'_{i} \in k^{\perar}$ with fields $k'_{i}$,
			define $k'^{\perar}$ to be the category of $k'$-algebras $k'' = \prod k''_{i}$
			with each factor $k''_{i} \in k_{i}'^{\perar}$ with $k'$-algebra homomorphisms.
	\end{enumerate}
\end{Def}

\begin{Prop} \label{prop: slice of gen rational algebras}
	For any $k' \in k^{\perar}$,
	the category $k'^{\perar}$ is canonically equivalent to
	the category of morphisms $k' \to k''$ from $k'$ in $k^{\perar}$.
\end{Prop}

\begin{proof}
	A perfect field extension of a perfect field extension is a perfect field extension.
	This implies the result.
\end{proof}

A similar statement does not hold for the category of ind-rational $k$-algebras $k^{\ind\rat}$
(see the second paragraph after \cite[Definition 2.1.3]{Suz13}).
An \'etale algebra over an object of $k^{\perar}$ is again in $k^{\perar}$.
The tensor product $k_{2} \tensor_{k_{1}} k_{3}$ of morphisms in $k^{\perar}$
does not belong to $k^{\perar}$ in general,
but it does if either $k_{2}$ or $k_{3}$ is \'etale over $k_{1}$.
Now we define the site $\Spec k^{\perar}_{\et}$.

\begin{Def} \mbox{}
	\begin{enumerate}
		\item
			For any $k' \in k^{\perar}$,
			we put the \'etale topology on (the opposite category of) the category $k'^{\perar}$
			and denote the resulting site by $\Spec k'^{\perar}_{\et}$.
			That is, a covering of $k'' \in k'^{\perar}$ is a finite family of \'etale $k''$-algebras $\{k''_{i}\}$
			such that $\prod k''_{i}$ is faithfully flat over $k''$.
		\item
			We denote the cohomology functor for $\Spec k^{\perar}_{\et}$ at the object $k'$
			by $H^{n}(k'_{\et}, \var)$, with derived categorical version $R \Gamma(k'_{\et}, \var)$.
		\item
			We denote the sheaf-Hom functor
			$\sheafhom_{\Spec k^{\perar}_{\et}}$ for $\Spec k^{\perar}_{\et}$
			by $\sheafhom_{k^{\perar}_{\et}}$.
	\end{enumerate}
\end{Def}

As above, we are not always strictly rigorous about the distinction
between the algebra $k'$ and the corresponding affine scheme $\Spec k'$ in this paper.
The context should make it clear.

The general rule to denote a site in this paper is that
the upper script (such as ``$\perar$'') denotes the type of objects of the underlying category
and the lower script (such as ``$\et$'') denotes the topology.

\begin{Prop} \label{prop: grational site compatible with localization}
	For any $k' \in k^{\perar}$,
	the site $\Spec k'^{\perar}$ is canonically equivalent to
	the localization $\Spec k^{\perar} / k'$ of the site $\Spec k^{\perar}$ at the object $k'$.
\end{Prop}

\begin{proof}
	This follows from Proposition \ref{prop: slice of gen rational algebras}.
\end{proof}

For any perfect $k$-scheme (resp.\ perfect $k$-group scheme) $X$,
we denote by the same symbol $X$ to also mean
the sheaf of sets (resp.\ groups) on $\Spec k^{\perar}_{\et}$ represented by $X$,
which is described as follows.

\begin{Prop} \label{prop: scheme as sheaf on rational etale site}
	Any perfect $k$-scheme $X$ as a sheaf of sets on $\Spec k^{\perar}_{\et}$
	is given by the disjoint union of its points (identified with the spectra of the residue fields).
	As a presheaf of sets,
	this disjoint union sheaf in the \'etale (or Zariski) topology may be described as
	the filtered union of finite sets of points of $X$.
\end{Prop}

\begin{proof}
	Any morphism $\Spec k' \to X$ from the spectrum of a perfect field extension $k'$ of $k$
	factors uniquely through a point of $X$.
	If $k' \in k^{\perar}$, then any morphism $\Spec k' \to X$
	factors uniquely through a finite set of points of $X$.
	These show the proposition.
\end{proof}

We recall the site $\Spec k^{\perf'}_{\pro\fppf}$ defined in \cite[Remark 3.8.4]{Suz13}
(which is a variant of the site $\Spec k^{\perf}_{\pro\fppf}$ defined in \cite[Section 3.1]{Suz13}).

\begin{Def} \mbox{}
	\begin{enumerate}
		\item
			Define $k^{\perf'}$ to be the category of quasi-compact quasi-separated perfect $k$-schemes
			with $k$-scheme morphisms.
		\item
			A morphism $Y \to X$ in $k^{\perf'}$ is said to be \emph{flat of finite presentation}
			(in the perfect sense) if it is the perfection
			of a $k$-morphism $Y_{0} \to X$ flat of finite presentation in the usual sense.
		\item
			A morphism $Y \to X$ in $k^{\perf'}$ is said to be \emph{flat of profinite presentation}
			if it can be written as the inverse limit $\invlim Y_{\lambda} \to X$
			of a filtered inverse system of morphisms $Y_{\lambda} \to X$ in $k^{\perf'}$
			flat of finite presentation (in the above perfect sense)
			with affine transition morphisms $Y_{\mu} \to Y_{\lambda}$.
		\item
			A \emph{faithfully flat morphism of (pro)finite presentation}
			is, by definition, a flat morphism of (pro)finite presentation that is surjective.
		\item
			We define the site $\Spec k^{\perf'}_{\pro\fppf}$ to be the category $k^{\perf'}$
			where a covering $\{X_{i} \to X\}$ is a finite jointly surjective family of
			morphisms $X_{i} \to X$ flat of profinite presentation.
		\item
			For $X \in k^{\perf'}$, we denote the localization of $\Spec k^{\perf'}_{\pro\fppf}$ at $X$
			by $X^{\perf'}_{\pro\fppf}$.
			
	\end{enumerate}
\end{Def}

See \cite[Remark 3.8.4]{Suz13} and \cite[Section 3.1]{Suz13} for the details
about $\Spec k^{\perf'}_{\pro\fppf}$
(see also \cite[Appendix A]{Suz14}).
Restricting the objects of the underlying category to affine schemes,
we have the corresponding pro-fppf site $\Spec k^{\perf}_{\pro\fppf}$ of perfect affine $k$-schemes.
The morphism of sites $\Spec k^{\perf'}_{\pro\fppf} \to \Spec k^{\perf}_{\pro\fppf}$
defined by the inclusion functor on the underlying categories
induces an equivalence on the topoi
by the same proof as \cite[Proposition (A.4)]{Suz14}.
We only use $\Spec k^{\perf'}_{\pro\fppf}$ in this paper,
though \cite{Suz13} uses $\Spec k^{\perf}_{\pro\fppf}$ and we use results from \cite{Suz13}.

We relate $\Spec k^{\perar}_{\et}$ to $\Spec k^{\perf'}_{\pro\fppf}$.

\begin{Def}
	We denote by
		\[
				\artin{h}
			\colon
				\Spec k^{\perf'}_{\pro\fppf}
			\to
				\Spec k^{\perar}_{\et}
		\]
	the premorphism of sites defined by the inclusion functor on the underlying categories.
\end{Def}

In \cite[Section 3.5]{Suz13}, a similar closely related premorphism
$\Spec k^{\perf}_{\pro\fppf} \to \Spec k^{\ind\rat}_{\et}$
to the ind-rational \'etale site was denoted by $h$,
though we do not technically need $h$ in this paper.
We generally put the accent symbol $\acute{\;}$ to distinguish
objects for the ind-rational \'etale site and objects for the perfect artinian \'etale site.
We need to clearly distinguish and compare these two sites when we cite \cite{Suz13}.
A perfect $k$-scheme $X$ viewed as a sheaf on $\Spec k^{\perar}_{\et}$ is
nothing but $\artin{h}_{\ast} X$.

Just as $h$ is not a morphism of sites (\cite[Proposition 3.2.3]{Suz13}),
neither is $\artin{h}$ by the same reason:

\begin{Prop} \label{prop: not a morphism of sites}
	The pullback $\artin{h}^{\ast \set}$ for sheaves of sets is \emph{not exact}.
	More explicitly, the natural morphism
	\[
			\artin{h}^{\ast \set}(\mathbb{A}_{k}^{2})
		\to
			(\artin{h}^{\ast \set} \mathbb{A}_{k}^{1})^{2}
	\]
	in $\Set(k^{\perf'}_{\pro\fppf})$ is not an isomorphism,
	where $\mathbb{A}_{k}^{n}$ is the perfection of affine $n$-space over $k$.
\end{Prop}

\begin{proof}
	By Proposition \ref{prop: scheme as sheaf on rational etale site},
	the sheaf $\artin{h}^{\ast \set}(\mathbb{A}_{k}^{2})$ is
	the disjoint union of points of the $k$-scheme $\mathbb{A}_{k}^{2}$.
	Also, the sheaf $(\artin{h}^{\ast \set} \mathbb{A}_{k}^{1})^{2}$ is the disjoint union of
	the $k$-scheme fiber products $x \times_{k} y$, where $x, y \in \mathbb{A}_{k}^{1}$.
	If $x = y$ is the generic point of $\mathbb{A}_{k}^{1}$,
	then $x \times_{k} y$ is not a point and hence not contained in the image of the morphism in question.
\end{proof}

There is a certain functoriality available for $\artin{h}$:

\begin{Prop} \label{prop: morphism h compatible with localization}
	Let $k'$ be a perfect field extension of $k$.
	Consider the premorphism $\artin{h}$ with $k$ replaced by $k'$,
	and denote it by $\artin{h}_{k'} \colon \Spec k'^{\perf'}_{\pro\fppf} \to \Spec k'^{\perar}_{\et}$.
	This agrees with the restriction
	$\artin{h}|_{k'} \colon \Spec k'^{\perf'}_{\pro\fppf} \to \Spec k^{\perar}_{\et} / k'$ of $\artin{h}$
	under the identification given in
	Proposition \ref{prop: grational site compatible with localization}.
\end{Prop}

\begin{proof}
	Obvious.
\end{proof}

We study the derived pullback functor $L \artin{h}^{\ast}$.
First, it does nothing on \'etale group schemes.
More precisely:

\begin{Prop} \label{prop: derived pull of etale groups}
	Let $G$ be a commutative \'etale group scheme over $k$.
	Consider the natural morphism $L \artin{h}^{\ast} G \to \artin{h}^{\ast} G$
	and the counit of adjunction $\artin{h}^{\ast} G = \artin{h}^{\ast} \artin{h}_{\ast} G \to G$.
	Their composite $L \artin{h}^{\ast} G \to G$ is an isomorphism
	(or equivalently, $L_{n} \artin{h}^{\ast} G = 0$ for $n \ge 1$
	and $\artin{h}^{\ast} G \isomto G$).
\end{Prop}

\begin{proof}
	Let $f \colon \Spec k^{\perf'}_{\pro\fppf} \to \Spec k_{\et}$ and
	$g \colon \Spec k^{\perar}_{\et} \to \Spec k_{\et}$ be
	the morphisms of sites defined by the inclusion functors on the underlying categories.
	Then $f = g \compose \artin{h}$.
	Hence $f^{\ast} = L \artin{h}^{\ast} g^{\ast}$
	by Proposition \ref{prop: composite of derived pullback}.
	Since $G$ is \'etale over $k$, we have $g^{\ast} G = g^{\ast} g_{\ast} G \isomto G$
	and $f^{\ast} G = f^{\ast} f_{\ast} G \isomto G$.
	Applying $f^{\ast} = L \artin{h}^{\ast} g^{\ast}$ to $G$, we get the result.
\end{proof}

Now we study $L \artin{h}^{\ast}$ applied to pro-algebraic groups.
Recall from the Notation part of Section \ref{sec: Introduction}
that we use the symbols $\Ind$ and $\Pro$ to denote
the ind-category and the pro-category constructions, respectively.

\begin{Def} \mbox{}
	\begin{enumerate}
		\item
			Define $\Alg / k$ be the category of perfections
			of commutative algebraic groups over $k$
			with group scheme morphisms over $k$.
		\item
			Define $\Pro' \Alg / k \subset \Pro \Alg / k$ be the full subcategory
			consisting of extensions of perfections of abelian varieties by perfect affine group schemes.
		\item
			For any $G \in \Pro' \Alg / k$ (or $\Pro \Alg / k$),
			its group of geometric connected components is denoted by $\pi_{0}(G)$,
			which is a pro-finite-\'etale group scheme over $k$
			(see the paragraph after \cite[Equation (2.1.1)]{Suz14}).
	\end{enumerate}
\end{Def}

The category $\Pro' \Alg / k$ was previously denoted by $\Pro' \Alg' / k$ in \cite[Remark 3.8.4]{Suz13}.
The category $\Pro \Alg / k$ is the (abelian) category of pro-algebraic groups
in the sense of Serre \cite{Ser60};
see the paragraph after \cite[Equation (2.1.1)]{Suz14} for more details on Serre's category.
Any object of $\Pro' \Alg / k$ is representable in $k^{\perf'}$.
We view $\Pro' \Alg / k \subset \Ab(k^{\perf'}_{\pro\fppf})$,
which is an exact embedding by \cite[Proposition (2.1.2) (e)]{Suz14}.

\begin{Def} \mbox{} \label{def: group with finite components and generic point}
	\begin{enumerate}
		\item
			Define $\Pro'_{\fc} \Alg / k$ to be the full subcategory of $\Pro' \Alg / k$
			consisting of objects $G$ with finite $\pi_{0}(G)$.
		\item
			The disjoint union of the generic points
			of the irreducible components of $G \in \Pro'_{\fc} \Alg / k $ is
			denoted by $\xi_{G} \in k^{\perar}$.
	\end{enumerate}
\end{Def}

The category $\Pro'_{\fc} \Alg / k$ is closed under cokernel,
but not under kernel and hence not abelian.
For instance, the kernel of multiplication by $l \ne p$
on the perfection of the connected affine group $\Gm^{\N}$ is not finite.

Recall from \cite[Proposition (2.3.4)]{Suz14} that
the Yoneda functor induces a fully faithful embedding
from the ind-category $\Ind \Pro \Alg / k = \Ind(\Pro \Alg / k)$
to $\Ab(k^{\perf'}_{\pro\fppf})$,
which itself induces a fully faithful embedding
$D^{b}(\Ind \Pro \Alg / k) \into D^{b}(k^{\perf'}_{\pro\fppf})$.
We define a slightly larger category than $\Pro'_{\fc} \Alg / k$
so that we can simultaneously treat non-finite \'etale group schemes
such as the discrete group scheme $\Z$.

\begin{Def} \label{def: extended ind of groups with finite components}
	Define $\mathcal{E}_{k}$ to be
	the full subcategory of $\Ab(k^{\perf'}_{\pro\fppf})$
	consisting of objects $G$ that can be written as an extension
	$0 \to G' \to G \to G'' \to 0$,
	where $G' \in \Ind \Pro'_{\fc} \Alg / k$
	and $G''$ an \'etale group scheme over $k$
	and the morphism $G \to G''$ is surjective
	(not only in the pro-fppf topology but also) in the \'etale topology.
\end{Def}

The category $\mathcal{E}_{k}$ contains perfections of smooth group schemes over $k$.
As above, an object $G \in \Pro' \Alg / k$
viewed as an object of $\Ab(k^{\perar}_{\et})$ (or equivalently, $\artin{h}_{\ast} G$) is
denoted by the same symbol $G$.
We extend this convention to $G \in \mathcal{E}_{k}$,
writing $\artin{h}_{\ast} G$ simply as $G$.

\begin{Thm} \label{thm: derived pull of proalg group with finite components}
	Let $G \in \mathcal{E}_{k}$.
	Then the morphism $L \artin{h}^{\ast} G \to G$ in $D(k^{\perf'}_{\pro\fppf})$
	defined as in Proposition \ref{prop: derived pull of etale groups}
	is an isomorphism.
\end{Thm}

The proof of this theorem will finish at the end of
Section \ref{sec: Acyclicity of the pullback of Mac Lane resolution}.

%%%%%%%%%%%%%%%%%%%%%%%%%%%%%%%%%%%%%%%%%%%%%%%%%%%%%%%%%%%%%%%%%%%%%%%%%%%%%

\section{Review of Mac Lane's resolution}
\label{sec: Review of Mac Lane resolution}

To compute the derived pullback,
we need Mac Lane's canonical resolution of abelian groups \cite{ML57}.
We merely provide notation for Mac Lane's resolution
and list its properties that we will use later.
For the definition itself, see \cite{ML57}.
Also see \cite[Chapter 13]{Lod98} for a more accessible account.
What we need are summarized in \cite[Section 3.4]{Suz13}.

As one can see from the proof of
Proposition \ref{prop: Mac Lane calculates derived pull} in the next section,
the key point of Mac Lane's resolution of an abelian group $G$ is that
each term is essentially built up of terms of the form $\Z[G]$ (the group ring of $G$),
and \emph{not} of the form $\Z[\Z[G]], \Z[\Z[\Z[G]]]$ etc.\
that one typically needs for simplicial resolutions,
so that the higher derived pullback of each term by a premorphism of sites
vanishes under some representability condition.
Phrased differently,
Ext groups $\Ext^{n}(G, \var)$ for a sheaf of groups $G$ on a site can be essentially described by
cohomology groups $H^{n}(G, \var) = \Ext^{n}(\Z[G], \var)$ of $G$
if $G$ satisfies some representability condition.
This is important when we have an exactness property of a relevant pullback functor
only for representable presheaves
(i.e.\ when we have a premorphism of sites that is not a morphism of sites).

We need symbols for Mac Lane's resolution and the related construction, the cubical construction.

\begin{Def} \mbox{}
	\begin{enumerate}
		\item
			We denote the free abelian group generated by a set $X$ by $\Z[X]$.
		\item
			Let $G \mapsto Q'(G), Q(G), M(G)$ be the (non-additive) functors
			that assign (homologically non-negatively graded) chain complexes to abelian groups $G$
			defined in \cite[\S 4, \S 7]{ML57}
			(where the base ring $\Lambda$ is taken to be $\Z$).
	\end{enumerate}
\end{Def}

See also \cite[\S 13.2 and E.13.2.1]{Lod98} for $Q'$ and $Q$.
In the notation of \cite[Lemma 13.2.12]{Lod98},
$M(G)$ is given by the two-sided bar construction $B(\Z, Q(\Z), Q(G))$
(where we and \cite{ML57} do not assume $G$ to be finitely generated free).
We will use the following properties.

\begin{Prop} \label{prop: cubical and Mac Lane} \mbox{}
	\begin{enumerate}
		\item \label{item: structure of each term of cubical}
			The $n$-th term $Q'_{n}(G)$ of $Q'(G)$ for any $n \ge 0$ is given by $\Z[G^{2^{n}}]$.
		\item \label{item: cubical is a quotient}
			The complex $Q(G)$ is a functorial quotient of $Q'(G)$ by a subcomplex.
		\item \label{item: quotient to cubical admits a section}
			For each $n \ge 0$, the quotient map $Q_{n}'(G) \onto Q_{n}(G)$ admits
			a functorial splitting $s_{n} \colon Q_{n}(G) \into Q_{n}'(G)$
			(\cite[Lemma 13.2.6]{Lod98} for example).
		\item \label{item: Mac Lane is a resolution}
			We have a functorial homomorphism $M_{0}(G) \to G$
			(which is given by $\Z[G] / \Z (0_{G}) \onto G$, $(g) \mapsto g$),
			and the complex
			$M(G) = (\cdots \overset{\partial}{\to} M_{1}(G) \overset{\partial}{\to} M_{0}(G))$
			is a functorial resolution of $G$
			(\cite[Th\'eor\`eme 6]{ML57}).
		\item \label{item: structure of Mac Lane without differentials}
			As a graded abelian group forgetting the differentials,
			$M(G)$ can be functorially written as $Q(G) \tensor_{\Z} \mathcal{B}$
			for some graded abelian group $\mathcal{B}$ that does not depend on $G$
			and whose $n$-th term is free for any $n$.
	\end{enumerate}
\end{Prop}

The group $\mathcal{B}$ is $\Bar{B}(0, Q(\Z), \eta_{Q})$
in Mac Lane's notation \cite[\S 7, Remarque 1]{ML57}
and $B(\Z, Q(\Z), \Z)$ as a two-sided bar construction.

We recall the splitting homotopy from the last two paragraphs of \cite[Section 3.4]{Suz13}.
See also \cite[\S 5, \S 8]{ML57} (resp.\ the proof of \cite[Theorem 11.2]{EM51}) for
the splitting homotopy with respect to additive projections (resp.\ arbitrary homomorphisms).

\begin{Def} \label{def: splitting homotopy}\mbox{}
	\begin{enumerate}
		\item
			Let $\varphi_{0}, \varphi_{1} \colon G \to H$ be any homomorphisms of abelian groups
			with sum $\varphi = \varphi_{0} + \varphi_{1}$ and let $n \ge 0$.
			We view $2 = \{0, 1\}$.
			Define a map $V \colon G^{2^{n}} \to H^{2^{n + 1}}$ by sending
			$(a_{i(1), \dots, i(n)})_{0 \le i(1), \dots, i(n) \le 1}$
			to
				\[
					\Bigl(
						\varphi_{0}(a_{i(2), \dots, i(n + 1)}) \text{ for } i(1) = 0,\,
						\varphi_{1}(a_{i(2), \dots, i(n + 1)}) \text{ for } i(1) = 1
					\Bigr)_{0 \le i(1), i(2), \dots, i(n + 1) \le 1}.
				\]
			In other words, $V = (\varphi_{0}, \varphi_{1}) \colon G^{2^{n}} \to H^{2^{n}} \times H^{2^{n}}$
			with respect to the decomposition $2^{n + 1} = 2 \times 2^{n}$.
		\item
			This extends to a homomorphism $V \colon Q'_{n}(G) \to Q'_{n + 1}(H)$,
			which factors through the quotient $V \colon Q_{n}(G) \to Q_{n + 1}(H)$.
			Define a homomorphism $V \colon M_{n}(G) \to M_{n + 1}(H)$ by
			$V \tensor \id$ on $Q(G) \tensor_{\Z} \mathcal{B}$.
		\item
			The homomorphisms $\varphi_{0}, \varphi_{1}, \varphi$ induce homomorphisms
			$\varphi_{0}, \varphi_{1}, \varphi \colon Q'(G) \to Q'(H)$ of complexes by functoriality.
			We have similar homomorphisms of complexes for $Q$ and $M$.
			Define a homomorphism of complexes $T \colon Q'(G) \to Q'(H)$ by $\varphi_{0} + \varphi_{1}$,
			and similarly $T \colon Q(G) \to Q(H)$ and $T \colon M(G) \to M(H)$.
	\end{enumerate}
\end{Def}

Note that $\varphi \ne \varphi_{0} + \varphi_{1}$ for $Q'$, $Q$ and $M$
since they are non-additive functors.
Here is the key property,
which says that $Q'$, $Q$ and $M$ are additive up to the homotopy $V$
(see the cited references above for the proof):

\begin{Prop}
	We have
		\[
			T - \varphi = \partial V + V \partial
		\]
	as $Q'(G) \to Q'(H)$, $Q(G) \to Q(H)$ and $M(G) \to M(H)$,
	where $\partial$ denotes the differentials of these complexes.
\end{Prop}

The above constructions generalize to (pre)sheaves of groups $G$ by functoriality.

\begin{Def}
	Let $S$ be a site.
	\begin{enumerate}
		\item
			For a presheaf of abelian groups $G$, define
				\begin{gather*}
					\Z_{\pre}[G](X) = \Z[G(X)], \quad
					Q_{\pre}'(G)(X) = Q'(G(X)),
					\\
					Q_{\pre}(G)(X) = Q(G(X)), \quad
					M_{\pre}(G)(X) = M(G(X)),
				\end{gather*}
			for $X \in S$.
		\item
			For $G \in \Ab(S)$ (that is, a sheaf of abelian groups),
			sheafification produces corresponding sheaves $\Z[G]$, $Q(G)$, $Q'(G)$ and $M(G)$.
	\end{enumerate}
\end{Def}

The subscript $\pre$ for the presheaf constructions means ``pre''.
Of course $T$ and $V$ extend to the (pre)sheaf setting,
but we do not need this extension in this paper.

%%%%%%%%%%%%%%%%%%%%%%%%%%%%%%%%%%%%%%%%%%%%%%%%%%%%%%%%%%%%%%%%%%%%%%%%%%%%%

\section{Mac Lane's resolution and derived pullback}
\label{sec: Mac Lane resolution and derived pullback}

We relate Mac Lane's resolution to derived pullback functors.
Let $f \colon S' \to S$ be a continuous map of sites
and $f^{-1}$ the underlying functor on the underlying categories.
Let $G' \in \Ab(S')$.
For any $X \in S$, the $X$-valued points of the complexes $M_{\pre}(f_{\ast} G')$ and $f_{\ast} M_{\pre}(G')$
(where $f_{\ast}$ is applied term-wise)
both give $M(G'(f^{-1} X))$.
Hence $M_{\pre}(f_{\ast} G') \cong f_{\ast} M_{\pre}(G')$ as complexes of presheaves on $S$.
With sheafification, we obtain a canonical morphism
$M(f_{\ast} G') \to f_{\ast} M(G')$ of complexes of sheaves on $S$.
By adjunction, we obtain a canonical morphism
$f^{\ast} M(f_{\ast} G') \to M(G')$ of complexes of sheaves on $S'$.
Composing it with the morphism $M(G') \to G'$,
we obtain a canonical morphism $f^{\ast} M(f_{\ast} G') \to G'$.
In other words, we have a complex
	\[
			\dots
		\to
			f^{\ast} M_{1}(f_{\ast} G')
		\to
			f^{\ast} M_{0}(f_{\ast} G')
		\to
			G'
		\to
			0
	\]
in $\Ab(S')$.

We can ask whether or not
the complex $f^{\ast} M(f_{\ast} G')$ gives a resolution of $G'$ in this manner.
When $f$ is a premorphism of sites,
this question is closely related to whether
the morphism $L f^{\ast} f_{\ast} G' \to G'$ in $D(S')$ is an isomorphism or not.
To see this relation, we need the fact that Mac Lane's resolution calculates $L f^{\ast}$
under a certain representability condition:

\begin{Prop} \label{prop: Mac Lane calculates derived pull}
	Let $f \colon S' \to S$ be a premorphism of sites defined by pretopologies
	and $G \in \Ab(S)$.
	Assume that $G$ as a sheaf of sets is
	the sheafification of a filtered direct limit of representable presheaves.
	Then the natural morphism
		\[
			L f^{\ast} M(G) \to f^{\ast} M(G)
		\]
	in $D(S')$ is an isomorphism.
\end{Prop}

\begin{proof}
	The spectral sequence
		\[
				E_{1}^{i j}
			=
				L_{- j} f^{\ast} M_{- i}(G)
			\converges
				H^{i + j} L f^{\ast} M(G),
		\]
	shows that it is enough to show $L_{j} f^{\ast} M_{i}(G) = 0$ for any $j \ge 1$ and $i \ge 0$.
	Furthermore, it is enough to show that $L_{j} f^{\ast} \Z[G] = 0$
	by the structure of $M_{i}(G)$
	(Proposition \ref{prop: cubical and Mac Lane}
	(\ref{item: structure of each term of cubical}),
	(\ref{item: cubical is a quotient}),
	(\ref{item: quotient to cubical admits a section}),
	(\ref{item: structure of Mac Lane without differentials})).
	Being a left adjoint, $f^{\ast}$ commutes with direct limits.
	Hence by \cite[Lemma 3.7.2]{Suz13} and \cite[Corollary 14.4.6 (ii)]{KS06},
	we know that $L_{j} f^{\ast}$ commutes with filtered direct limits.
	Hence, by the assumption on $G$,
	the statement reduces to the fact \cite[Lemma 3.7.2]{Suz13}
	that $L_{j} f^{\ast} \Z[X] = 0$ for $X \in S$.
\end{proof}

Using this, we obtain the desired relation:

\begin{Prop} \label{prop: derived pull and pull Mac Lane resolution}
	Let $f \colon S' \to S$ be a premorphism of sites defined by pretopologies
	and $G' \in \Ab(S')$.
	Assume that $f_{\ast} G'$ as a sheaf of sets is
	the sheafification of a filtered direct limit of representable presheaves.
	Then the morphism $L f^{\ast} f_{\ast} G' \to G'$ is an isomorphism in $D(S')$
	if and only if $f^{\ast} M(f_{\ast} G')$ is a resolution of $G'$.
\end{Prop}

\begin{proof}
	We have a commutative diagram
		\[
			\begin{CD}
					L f^{\ast} M(f_{\ast} G')
				@>>>
					f^{\ast} M(f_{\ast} G')
				@>>>
					M(G')
				\\
				@VV \wr V
				@VVV
				@V \wr VV
				\\
					L f^{\ast} f_{\ast} G'
				@>>>
					f^{\ast} f_{\ast} G'
				@>>>
					G'
			\end{CD}
		\]
	in $D(S')$.
	Hence the result follows from the previous proposition.
\end{proof}

Note that $G' = \Z[X']$ for $X' \in S'$ does not always satisfy the assumption of the proposition.
The description of $\Z[X']$ as a sheaf of sets involves quotients of powers of $X'$ by some equivalence relations,
which are not \emph{filtered} direct limits.

To prove Theorem \ref{thm: derived pull of proalg group with finite components},
the key step will be to show that $\artin{h}^{\ast} M(G)$ is a resolution of $G \in \Pro_{\fc}' \Alg / k$
in $\Ab(k^{\perf'}_{\pro\fppf})$,
which we will do in the next section.

%%%%%%%%%%%%%%%%%%%%%%%%%%%%%%%%%%%%%%%%%%%%%%%%%%%%%%%%%%%%%%%%%%%%%%%%%%%%%

\section{Acyclicity of the pullback of Mac Lane's resolution}
\label{sec: Acyclicity of the pullback of Mac Lane resolution}

Let $G \in \Pro_{\fc}' \Alg / k$.
Let $L$ be a finitely generated free abelian group.
Denote the sheaf-Hom, $\sheafhom_{k^{\perf'}_{\pro\fppf}}(L, G)$, by $[L, G] \in \Pro_{\fc}' \Alg / k$.
Let $X \in k^{\perf'}$ and set $Y = X \times_{k} [L, G] \in k^{\perf'}$.
Let $\varphi \colon L \to G(X)$ be a homomorphism.
Its composite with the homomorphism $G(X) \to G(Y)$ induced by the first projection $Y \to X$
is still denoted by $\varphi$.
The natural evaluation homomorphism $L \to G([L, G])$ is denoted by $\varphi_{0}$.
Therefore $\varphi(a)$ for any $a \in L$ is a morphism $X \to G$ in $k^{\perf'}$
and $\varphi_{0}(a)$ is a morphism $[L, G] \to G$ in $\Pro_{\fc}' \Alg / k$.
The composite of $\varphi_{0}$ with the homomorphism $G([L, G]) \to G(Y)$
induced by the second projection $Y \to [L, G]$
is still denoted by $\varphi_{0}$.
Set $\varphi_{1} = \varphi - \varphi_{0} \colon L \to G(Y)$.
Since $G \in \Pro_{\fc}' \Alg / k$ is faithfully flat of profinite presentation over $k$,
the scheme $Y$ is faithfully flat of profinite presentation over $X$.
We recall the following fact from \cite[Lemma 3.6.2]{Suz13}
to create sufficiently many pro-fppf covers of $X$:

\begin{Prop} \label{prop: making profppf covers}
	Let $Z_{i} \to Y$ be morphisms in $k^{\perf'}$,
	$i = 1, \dots, n$, and
	let $Z = Z_{1} \times_{Y} \dots \times_{Y} Z_{n}$.
	Assume the following conditions for each $i$:
		\begin{itemize}
			\item
				$Z_{i} \to Y$ is flat of profinite presentation,
			\item
				the morphism $(Z_{i})_{x} \to Y_{x}$ on the fiber over any point $x \in X$ is dominant.
		\end{itemize}
	Then $Z$ also satisfies these two conditions.
	In particular, $Z / X$ is faithfully flat of profinite presentation.
\end{Prop}

\begin{proof}
	This follows from the fact that
	a flat base change of a dominant morphism is dominant.
\end{proof}

Recall that $\artin{h}$ is not a morphism of sites
(Proposition \ref{prop: not a morphism of sites}).
We note the structures of $\artin{h}_{\pre}^{\ast \set} G$
and $\artin{h}_{\pre}^{\ast} \Z_{\pre}[G]$.

\begin{Prop} \label{prop: perar pull is union of finite sets of points}
	Let $\artin{h}_{\pre}^{\ast}$ (resp.\ $\artin{h}_{\pre}^{\ast \set}$) be the pullback functors
	for presheaves of abelian groups (resp.\ sets) by $\artin{h}$.
	\begin{enumerate}
		\item
			The presheaf $\artin{h}_{\pre}^{\ast \set} G$ on $\Spec k^{\perf'}_{\pro\fppf}$
			is the filtered union of finite sets of points of $G$,
			which is a subpresheaf of $G$.
		\item
			The presheaf $\artin{h}_{\pre}^{\ast} \Z_{\pre}[G]$ on $\Spec k^{\perf'}_{\pro\fppf}$
			is the filtered union of $\Z_{\pre}[x]$ over the finite sets of points $x$ of $G$,
			which is a subpresheaf of $\Z_{\pre}[G]$.
	\end{enumerate}
\end{Prop}

\begin{proof}
	This follows from Proposition \ref{prop: scheme as sheaf on rational etale site}.
\end{proof}

The following proposition and its proof are a variant of \cite[Lemma 3.6.3]{Suz13},
with the ind-rational \'etale site replaced by the perfect artinian \'etale site.
The proposition allows us to pro-fppf locally ``generify'' sections of $G$ (and $G^{2}$).

\begin{Prop}
	Let $a \in L$.
	\begin{enumerate}
		\item \label{item: generify a section}
			There exist a scheme $Z \in k^{\perf'}$ and
			a $k$-morphism $Z \to Y$ satisfying the two conditions of
			Proposition \ref{prop: making profppf covers}
			such that the natural images $\varphi_{0}(a), \varphi_{1}(a) \in G(Z)$ are
			contained in the subset $(\artin{h}_{\pre}^{\ast \set} G)(Z)$.
		\item \label{item: further generify a generic section}
			If $\varphi(a) \in G(X)$ is contained in $(\artin{h}_{\pre}^{\ast \set} G)(X)$,
			then $Z$ can be taken so that
			the natural image $(\varphi_{0}(a), \varphi_{1}(a)) \in G^{2}(Z)$ is
			contained in $(\artin{h}_{\pre}^{\ast \set}(G^{2}))(Z)$.
	\end{enumerate}
\end{Prop}

\begin{proof}
	(\ref{item: generify a section})
	The element $\varphi_{0}(a)$ gives a morphism $[L, G] \to G$ in $\Pro_{\fc}' \Alg / k$,
	whose image $\Im(\varphi_{0}(a))$ is again in $\Pro_{\fc}' \Alg / k$.
	Hence its generic point $\xi_{\Im(\varphi_{0}(a))}$ is an object of $k^{\perar}$
	(Definition \ref{def: group with finite components and generic point}).
	Consider the following commutative diagram in $k^{\perf'}$ with a cartesian square:
		\[
			\begin{CD}
					\varphi_{0}(a)^{-1}(\xi_{\Im(\varphi_{0}(a))})
				@>> \varphi_{0}(a) >
					\xi_{\Im(\varphi_{0}(a))}
				\\
				@VV \text{incl} V
				@V \text{incl} VV
				\\
					[L, G]
				@> \varphi_{0}(a) >>
					\Im(\varphi_{0}(a))
				@> \text{incl} >>
					G,
			\end{CD}
		\]
	The bottom arrow in the square is faithfully flat of profinite presentation
	since it is a surjection of pro-algebraic groups.
	The right arrow is dominant flat of profinite presentation.
	Hence the left arrow is dominant flat of profinite presentation.
	We define $Z_{1} = X \times_{k} \varphi_{0}(a)^{-1}(\xi_{\Im(\varphi_{0}(a))})$.
	Then the natural morphism $Z_{1} \to Y$ satisfies
	the two conditions in Proposition \ref{prop: making profppf covers}
	by the same reasoning as the proof of Proposition \ref{prop: making profppf covers}.
	The natural image $\varphi_{0}(a) \in G(Z_{1})$ is a morphism
	$Z_{1} \to G$ that factors through
		$
				\xi_{\Im(\varphi_{0}(a))}
			\subset
				\artin{h}_{\pre}^{\ast\set} G
		$
	(Proposition \ref{prop: perar pull is union of finite sets of points}).
	Hence $\varphi_{0}(a) \in (\artin{h}_{\pre}^{\ast\set} G)(Z_{1})$.
	
	The morphism $\varphi \colon L \to G(X)$ defines a morphism
	$\Tilde{\varphi} \colon X \to [L, G]$.
	We have an automorphism of the $X$-scheme $Y = X \times_{k} [L, G]$ given by
	$(x, \psi) \leftrightarrow (x, \Tilde{\varphi}(x) - \psi)$.
	The composite of this with the morphism $\varphi_{0}(a) \colon Y \to G$
	is $\varphi(a) - \varphi_{0}(a) = \varphi_{1}(a)$.
	We define $Z_{2} \to Y$ to be the inverse image of
	the morphism $Z_{1} \to Y$ by this $X$-automorphism of $Y$.
	Then we have $\varphi_{1}(a) \in (\artin{h}_{\pre}^{\ast\set} G)(Z_{2})$
	by the previous paragraph
	and $Z_{2}$ satisfies the two conditions of Proposition \ref{prop: making profppf covers}.
	We define $Z = Z_{1} \times_{Y} Z_{2}$.
	Then we have $\varphi_{0}(a), \varphi_{1}(a) \in (\artin{h}_{\pre}^{\ast\set} G)(Z)$
	and $Z$ satisfies the two conditions of Proposition \ref{prop: making profppf covers}.
	
	(\ref{item: further generify a generic section})
	Assume that $\varphi(a) \in (\artin{h}_{\pre}^{\ast\set} G)(X)$.
	Consider the automorphism $(b, c) \leftrightarrow (b + c, b)$ of the group $G^{2}$,
	which maps $(\varphi_{0}(a), \varphi_{1}(a))$ to $(\varphi(a), \varphi_{0}(a))$.
	Hence it is enough to show that we can take $Z$ so that
	$(\varphi(a), \varphi_{0}(a)) \in (\artin{h}_{\pre}^{\ast\set}(G^{2}))(Z)$.
	The image $\Im(\varphi(a))$ of $\varphi(a) \colon X \to \artin{h}_{\pre}^{\ast\set} G \subset G$
	is an object of $k^{\perar}$
	(Proposition \ref{prop: perar pull is union of finite sets of points}).
	We have a faithfully flat morphism $\varphi(a) \colon X \onto \Im(\varphi(a))$
	of profinite presentation.
	Define $W = \Im(\varphi(a)) \times_{k} \Im(\varphi_{0}(a))$,
	which is the finite disjoint union of the fibers $W_{k'} \in \Pro'_{\fc} \Alg / k'$
	over the points $\Spec k'$ of $\Im(\varphi(a))$.
	In particular, its generic point $\xi_{W}$ is an object of $k^{\perar}$.
	We have a faithfully flat morphism
	$(\varphi(a), \varphi_{0}(a)) \colon X \times_{k} [L, G] \onto W$
	of profinite presentation.
	Consider the following commutative diagram with a cartesian square:
		\[
			\begin{CD}
					(\varphi(a), \varphi_{0}(a))^{-1}(\xi_{W})
				@>> (\varphi(a), \varphi_{0}(a)) >
					\xi_{W}
				\\
				@VV \text{incl} V
				@V \text{incl} VV
				\\
					X \times_{k} [L, G]
				@> (\varphi(a), \varphi_{0}(a)) >>
					W
				@> \text{incl} >>
					G^{2}
			\end{CD}
		\]
	We define $Z = (\varphi(a), \varphi_{0}(a))^{-1}(\xi_{W})$.
	Then $(\varphi(a), \varphi_{0}(a)) \in (\artin{h}_{\pre}^{\ast\set}(G^{2}))(Z)$
	by the same argument as above.
	The square in the above diagram can be split into the following two cartesian squares:
		\[
			\begin{CD}
					Z
				@>>>
					X \times_{\Im(\varphi(a))} \xi_{W}
				@>> \proj_{2} >
					\xi_{W}
				\\
				@VV \text{incl} V
				@VV \text{incl} V
				@V \text{incl} VV
				\\
					Y
				@> (\id, \varphi_{0}(a)) >>
					X \times_{k} \Im(\varphi_{0}(a))
				@> (\varphi(a), \id) >>
					W
			\end{CD}
		\]
	The bottom two arrows are faithfully flat of profinite presentation.
	The third vertical arrow is dominant flat of profinite presentation.
	By pulling back the left square by a point of $X$,
	we see that the morphism $Z \to Y$ satisfies the two conditions of
	Proposition \ref{prop: making profppf covers}.
\end{proof}

The above proposition extends to $\Z[G]$ and $\Z[G^{2}]$ in the following manner.

\begin{Prop}
	Let $t \in \Z[L]$.
	\begin{enumerate}
		\item
			There exist a scheme $Z \in k^{\perf'}$ and
			a $k$-morphism $Z \to Y$ satisfying the two conditions of
			Proposition \ref{prop: making profppf covers}
			such that the natural images $\varphi_{0}(t), \varphi_{1}(t) \in \Z[G(Z)] = \Z_{\pre}[G](Z)$ are
			contained in the subgroup $(\artin{h}_{\pre}^{\ast} \Z_{\pre}[G])(Z)$.
		\item
			If $\varphi(t) \in \Z[G(X)]$ is contained in $(\artin{h}_{\pre}^{\ast} \Z_{\pre}[G])(X)$,
			then $Z$ can be taken so that
			the natural image $(\varphi_{0}, \varphi_{1})(t) \in \Z[G^{2}(Z)]$ is
			contained in $(\artin{h}_{\pre}^{\ast} \Z_{\pre}[G^{2}])(Z)$.
	\end{enumerate}
\end{Prop}

\begin{proof}
	Write $t = \sum_{i = 1}^{n} m_{i} (a_{i})$,
	$m_{i} \in \Z$, $a_{i} \in L$,
	where $(a_{i})$ is the image of $a_{i}$ in $\Z[L]$.
	For any $i$, take $Z_{i}$ corresponding to $a_{i} \in L$
	given in the previous proposition (\ref{item: generify a section}).
	Then
		$
				\varphi_{0}(a_{i}), \varphi_{1}(a_{i})
			\in
				\Z_{\pre}[\artin{h}_{\pre}^{\ast\set} G](Z_{i})
			=
				(\artin{h}_{\pre}^{\ast} \Z_{\pre}[G])(Z_{i})
		$.
	Let $Z = Z_{1} \times_{Y} \dots \times_{Y} Z_{n}$.
	Then $Z / X$ is faithfully flat of profinite presentation
	by Proposition \ref{prop: making profppf covers},
	and we have $\varphi_{0}(t) = \sum m_{i} \varphi_{0}(a_{i}) \in (\artin{h}_{\pre}^{\ast} \Z_{\pre}[G])(Z)$
	and similarly $\varphi_{1}(t) \in (\artin{h}_{\pre}^{\ast} \Z_{\pre}[G])(Z)$.
	The statement for $(\varphi_{0}, \varphi_{1})(t) = \sum m_{i} (\varphi_{0}(a_{i}), \varphi_{1}(a_{i}))$
	is similar, using the previous proposition (\ref{item: further generify a generic section}).
\end{proof}

Now we consider the homomorphisms $V \colon Q'_{n}(L) \to Q'_{n + 1}(G(Y))$ and $T \colon Q'_{n}(L) \to Q'_{n}(G(Y))$
corresponding to $\varphi_{0}, \varphi_{1} \colon L \to G(Y)$
as defined in Definition \ref{def: splitting homotopy},
and similar morphisms for $Q$ and $M$.
The above proposition on $\Z[\var]$ extends to $Q', Q, M$
using the homomorphisms $T$ and $V$:

\begin{Prop}
	Let $F$ be one of the functors $Q'$, $Q$ or $M$.
	Let $n \ge 0$ and $t \in F_{n}(L)$.
	\begin{enumerate}
		\item \label{item: T generifies}
			There exist a scheme $Z \in k^{\perf'}$ and
			a $k$-morphism $Z \to Y$ satisfying the two conditions of
			Proposition \ref{prop: making profppf covers}
			such that the natural image $T(t) \in F_{n}(G(Z)) = F_{\pre, n}(G)(Z)$ is
			contained in the subgroup $(\artin{h}_{\pre}^{\ast} F_{\pre, n}(G))(Z)$.
		\item \label{item: V generifies}
			If $\varphi(t) \in F_{n}(G(X))$ is contained in $(\artin{h}_{\pre}^{\ast} F_{\pre, n}(G))(X)$,
			then $Z$ can be taken so that
			the natural image $V(t) \in F_{n + 1}(G(Z))$ is
			contained in $(\artin{h}_{\pre}^{\ast} F_{\pre, n + 1}(G))(Z)$.
	\end{enumerate}
\end{Prop}

\begin{proof}
	First, let $F = Q'$, so $F_{n} = \Z[(\var)^{2^{n}}]$
	(Proposition \ref{prop: cubical and Mac Lane}
	(\ref{item: structure of each term of cubical})).
	Then $t \in \Z[L^{2^{n}}]$.
	Applying the previous proposition to $L^{2^{n}}$ and $G^{2^{n}}$ instead of $L$ and $G$,
	we know that $\varphi_{0}(t), \varphi_{1}(t) \in \Z[G^{2^{n}}(Z)] = \Z_{\pre}[G^{2^{n}}](Z)$
	are contained in $(\artin{h}_{\pre}^{\ast} \Z_{\pre}[G^{2^{n}}])(Z)$.
	So is their sum $T(t)$.
	If $\varphi(t) \in \Z[G^{2^{n}}(X)]$ is contained in $(\artin{h}_{\pre}^{\ast} \Z_{\pre}[G^{2^{n}}])(X)$,
	then the same method shows that $Z$ can be taken so that
	the natural image $(\varphi_{0}, \varphi_{1})(t) = V(t) \in \Z[G^{2^{n + 1}}(Z)]$ is
	contained in $(\artin{h}_{\pre}^{\ast} \Z_{\pre}[G^{2^{n + 1}}])(Z)$.
	
	Next let $F = Q$.
	Let $s_{n} \colon Q_{n} \into Q'_{n}$ be the section to the quotient $Q'_{n} \onto Q_{n}$
	as in Proposition \ref{prop: cubical and Mac Lane}
	(\ref{item: quotient to cubical admits a section}).
	The case of $Q'$ above applied to $s_{n}(t) \in Q'_{n}(L)$ gives a scheme $Z \in k^{\perf'}$ and
	a $k$-morphism $Z \to Y$ satisfying the two conditions of
	Proposition \ref{prop: making profppf covers}
	such that the natural image $T(s_{n}(t)) \in Q'_{n}(G(Z)) = Q'_{\pre, n}(G)(Z)$ is
	contained in $(\artin{h}_{\pre}^{\ast} Q'_{\pre, n}(G))(Z)$.
	Taking the quotient, we know that $T(t) \in Q_{n}(G(Z)) = Q_{\pre, n}(G)(Z)$ is
	contained in $(\artin{h}_{\pre}^{\ast} Q_{\pre, n}(G))(Z)$.
	If $\varphi(t) \in Q_{n}(G(X))$ is contained in $(\artin{h}_{\pre}^{\ast} Q_{\pre, n}(G))(X)$,
	then $s_{n}(\varphi(t)) = \varphi(s_{n}(t)) \in Q'_{n}(G(X))$ is
	contained in $(\artin{h}_{\pre}^{\ast} Q'_{\pre, n}(G))(X)$.
	Hence the case of $Q'$ implies that $Z$ can be taken so that
	the natural image $V(s_{n}(t)) \in Q'_{n + 1}(G(Z))$ is
	contained in $(\artin{h}_{\pre}^{\ast} Q'_{\pre, n + 1}(G))(Z)$.
	Thus $V(t) \in Q_{n + 1}(G(Z))$ is
	contained in $(\artin{h}_{\pre}^{\ast} Q_{\pre, n + 1}(G))(Z)$.
	
	Finally, let $F = M$.
	For each $m \ge 0$, take a basis $J_{m}$ of the free abelian group $\mathcal{B}_{m}$
	(Proposition \ref{prop: cubical and Mac Lane}
	(\ref{item: structure of Mac Lane without differentials}))
	and write $\mathcal{B}_{m} \cong \Z^{\bigoplus J_{m}}$.
	Then $M_{n} \cong \bigoplus_{m} Q_{m}^{\bigoplus J_{n - m}}$
	since $M = Q \tensor_{Z} \mathcal{B}$ as graded abelian groups.
	In this decomposition, $T$ and $V$ on $M_{n}$ are given by
	the term-wise applications of $T$ and $V$ on the $Q_{m}$.
	Hence the case of $F = Q$ implies the case of $F = M$.
\end{proof}

\begin{Prop}
	The complex $\artin{h}^{\ast} M(G)$ in $\Ab(k^{\perf'}_{\pro\fppf})$ is a resolution of $G$.
\end{Prop}

\begin{proof}
	We know that $\artin{h}^{\ast} M(G)$ is the pro-fppf sheafification of $\artin{h}_{\pre}^{\ast} M_{\pre}(G)$.
	We want to show that the complex
		\[
				\cdots
			\to
				\artin{h}^{\ast} M_{\pre, 1}(G)
			\to
				\artin{h}^{\ast} M_{\pre, 0}(G)
			\to
				G
			\to
				0
		\]
	in $\Ab(k^{\perf'}_{\pro\fppf})$ sheafifies to an acyclic complex.
	Let $n \ge -1$ and $X \in k^{\perf'}$.
	Let $u \in (\artin{h}_{\pre}^{\ast} M_{\pre, n}(G))(X)$ with $\partial u = 0$ if $n \ge 1$;
	$u \in (\artin{h}_{\pre}^{\ast} M_{\pre, 0}(G))(X)$ whose natural image in $G(X)$ is zero if $n = 0$;
	and $u \in G(X)$ if $n = -1$.
	Since $M(G(X))$ is a resolution of $G(X)$,
	there exists an element $u' \in M_{n + 1}(G(X))$ such that:
	$\partial u' = u$ if $n \ge 0$;
	and the natural image of $u'$ in $G(X)$ is $u$ if $n = -1$.
	Take a finitely generated free abelian group $L$, a homomorphism $\varphi \colon L \to G(X)$
	and an element $t' \in M_{n + 1}(L)$
	such that $\varphi(t') = u'$.
	Then $\varphi(\partial t') = u$,
	where $\partial t'$ in the case $n = -1$ is understood to be the natural image of $t'$ in $L$.
	Applying the previous proposition (\ref{item: T generifies}) to $t'$
	and (\ref{item: V generifies}) to $\partial t'$,
	we know that there exist a scheme $Z \in k^{\perf'}$ and
	a $k$-morphism $Z \to Y$ satisfying the two conditions of
	Proposition \ref{prop: making profppf covers} such that
	$T(t') \in M_{n + 1}(G(Z)) = M_{\pre, n + 1}(G)(Z)$ is
	contained in $(\artin{h}_{\pre}^{\ast} M_{\pre, n}(G))(Z)$
	and $V(\partial t') \in M_{n + 1}(G(Z))$ is
	contained in $(\artin{h}_{\pre}^{\ast} M_{\pre, n + 1}(G))(Z)$.
	We have $T(t') - \varphi(t') = \partial(V(t')) + V(\partial(t'))$ in $M_{n + 1}(G(Z))$.
	Hence $u = \partial(T(t') - V(\partial t'))$ if $n \ge 0$
	and $u$ is the image of $T(t') - V(\partial t')$ in $G(Z)$ if $n = -1$.
	Since $T(t') - V(\partial t') \in (\artin{h}_{\pre}^{\ast} M_{\pre, n + 1}(G))(Z)$
	and $Z / X$ faithfully flat of profinite presentation,
	this proves the result.
\end{proof}

Now we prove Theorem \ref{thm: derived pull of proalg group with finite components}.

\begin{proof}[Proof of Theorem \ref{thm: derived pull of proalg group with finite components}]
	Write $G \in \mathcal{E}_{k}$ as an extension
	$0 \to G' \to G \to G'' \to 0$ as in
	Definition \ref{def: extended ind of groups with finite components}.
	Since $G \to G''$ is surjective in the \'etale topology,
	this short exact sequence remains exact in $\Ab(k^{\perar}_{\et})$ (after applying $\artin{h}_{\ast}$).
	Therefore we may assume that $G \in \Ind \Pro_{\fc}' \Alg / k$
	by Proposition \ref{prop: derived pull of etale groups}.
	The functor $L_{n} \artin{h}^{\ast}$ for any $n$ commutes with filtered direct limits
	as we saw in the proof of Proposition \ref{prop: Mac Lane calculates derived pull}.
	Hence we may furthermore assume that $G \in \Pro_{\fc}' \Alg / k$.
	We know that $\artin{h}_{\ast} G$ as a sheaf of sets on $\Spec k^{\perar}_{\et}$ is
	the filtered union of finite sets of points of $G$
	by Proposition \ref{prop: scheme as sheaf on rational etale site}.
	Hence we may apply Proposition \ref{prop: derived pull and pull Mac Lane resolution} to $G$.
	This implies, by the previous proposition,
	that $L \artin{h}^{\ast} G \isomto G$ in $D(k^{\perf'}_{\pro\fppf})$.
	This proves the theorem.
\end{proof}

%%%%%%%%%%%%%%%%%%%%%%%%%%%%%%%%%%%%%%%%%%%%%%%%%%%%%%%%%%%%%%%%%%%%%%%%%%%%%

\section{Consequences of the derived pullback theorem}

Theorem \ref{thm: derived pull of proalg group with finite components}
together with the following finishes the proof of Theorem \ref{thm: main, derived pull}.

\begin{Prop} \label{prop: EIP embeds into sheaves on perar site}
	The functor $\mathcal{E}_{k} \to \Ab(k^{\perar}_{\et})$
	given by restriction of $\artin{h}_{\ast}$ is fully faithful.
\end{Prop}

\begin{proof}
	Let $G, H \in \mathcal{E}_{k}$.
	We have $\artin{h}^{\ast} G \cong G$
	by Theorem \ref{thm: derived pull of proalg group with finite components}.
	As we are denoting $\artin{h}_{\ast} H$ as $H$, we have
		\[
				\Hom_{k^{\perar}_{\et}}(G, H)
			\cong
				\Hom_{k^{\perf'}_{\pro\fppf}}(\artin{h}^{\ast} G, H)
			\cong
				\Hom_{k^{\perf'}_{\pro\fppf}}(G, H).
		\]
\end{proof}

We frequently identify $\mathcal{E}_{k}$
with its image in $\Ab(k^{\perar}_{\et})$.
The functor $L \artin{h}^{\ast}$ behaves well on $\mathcal{E}_{k}$:

\begin{Prop} \label{prop: EIP objects are h compatible}
	Any object of the triangulated subcategory
	$\genby{\mathcal{E}_{k}}_{D(k^{\perar}_{\et})}$
	of $D(k^{\perar}_{\et})$ generated by $\mathcal{E}_{k}$
	(see the notation section in Section \ref{sec: Introduction}) is $\artin{h}$-compatible.
\end{Prop}

\begin{proof}
	By Definition \ref{def: f compatible, f acyclic},
	the statement means that for $G \in \mathcal{E}_{k}$, the morphism
	$L(\artin{h}|_{k'})^{\ast}(G|_{k'}) \to (L \artin{h}^{\ast} G)|_{k'}$
	is an isomorphism for any $k' \in k^{\perar}$.
	To show this, we may assume that $G \in \Pro_{\fc}' \Alg / k$ and $k'$ is a field.
	The restriction $G|_{k'} \in \Ab(k'^{\perar})$ is represented by
	the scheme-theoretic fiber product $G \times_{k} k'$.
	The restriction $\artin{h}|_{k'}$ is the premorphism $\artin{h}$ with $k$ replaced by $k'$
	by Proposition \ref{prop: morphism h compatible with localization}.
	Therefore the result follows from
	Theorem \ref{thm: derived pull of proalg group with finite components}
	for $G$ over $k$ and $G \times_{k} k'$ over $k'$.
\end{proof}

Theorem \ref{thm: derived pull of proalg group with finite components}
and Proposition \ref{prop: EIP embeds into sheaves on perar site}
are about a single object $G \in \mathcal{E}_{k}$.
In practice, we need to treat an object of $\genby{\mathcal{E}_{k}}_{D(k^{\perar}_{\et})}$.
Then the corresponding statements are slightly tricky to state
since the expression ``$L \artin{h}^{\ast} G \isomto G$''
for $G \in \genby{\mathcal{E}_{k}}_{D(k^{\perar}_{\et})}$
does not literally make sense
(the both sides live in different categories).
Also, the most practical category is
the full subcategory of $D^{b}(k^{\perar}_{\et})$ of objects with cohomologies in $\mathcal{E}_{k}$,
which is an additive subcategory of $\genby{\mathcal{E}_{k}}_{D(k^{\perar}_{\et})}$
but is not triangulated.
With these in mind,
we have the following version of
Theorem \ref{thm: derived pull of proalg group with finite components}
and Proposition \ref{prop: EIP embeds into sheaves on perar site}
for these categories.

\begin{Prop} \label{prop: L h pull preseves groups with finite components} \mbox{}
	\begin{enumerate}
		\item
			The functor $L \artin{h}^{\ast} \colon D(k^{\perar}_{\et}) \to D(k^{\perf'}_{\pro\fppf})$
			maps the subcategory $\genby{\mathcal{E}_{k}}_{D(k^{\perar}_{\et})}$
			to the subcategory $\genby{\mathcal{E}_{k}}_{D(k^{\perf'}_{\pro\fppf})}$.
		\item
			For any $G \in D^{b}(k^{\perar}_{\et})$ whose cohomologies are in $\mathcal{E}_{k}$,
			the spectral sequence
				\[
						E_{2}^{i j}
					=
						L_{-i} \artin{h}^{\ast} H^{j}(G)
					\converges
						H^{i + j}(L \artin{h}^{\ast} G)
				\]
			has zero terms for all $i \ne 0$,
			and hence induces an isomorphism
				\[
					H^{n}(G) \cong H^{n}(L \artin{h}^{\ast} G)
				\]
			in $\mathcal{E}_{k}$ for all $n$.
	\end{enumerate}
\end{Prop}

\begin{proof}
	The first statement follows from the second, which itself follows from
	Theorem \ref{thm: derived pull of proalg group with finite components}.
\end{proof}

The above does not claim that $\genby{\mathcal{E}_{k}}_{D(k^{\perar}_{\et})}$
and $\genby{\mathcal{E}_{k}}_{D(k^{\perf'}_{\pro\fppf})}$ are equivalent via $L \artin{h}^{\ast}$.
The best we can say is the above somewhat clumsy statement.

We recall the ind-rational \'etale site $\Spec k^{\ind\rat}_{\et}$ and
the ind-rational pro-\'etale site $\Spec k^{\ind\rat}_{\pro\et}$
from \cite[Definition 2.1.3]{Suz13} and \cite[Section 2.1]{Suz14}.

\begin{Def} \mbox{}
	\begin{enumerate}
		\item
			Define $k^{\rat}$ to be the category of finite products
			of perfections of finitely generated field extensions of $k$
			with $k$-algebra homomorphisms.
		\item
			Define $k^{\ind\rat}$ to be its ind-category,
			whose objects may be identified with $k$-algebras
			that is a filtered union of rational $k$-subalgebras.
		\item
			An \emph{\'etale covering} of $k' \in k^{\ind\rat}$ is
			a finite family $\{k'_{i}\}$ of \'etale $k'$-algebras
			(necessarily in $k^{\ind\rat}$; \cite[Proposition 2.1.2]{Suz13})
			such that $\prod k'_{i}$ is faithfully flat over $k'$.
		\item
			Define the site $\Spec k^{\ind\rat}_{\et}$
			to be the category $k^{\ind\rat}$ with this class of coverings.
		\item
			A \emph{pro-\'etale covering} of $k' \in k^{\ind\rat}$ is
			a finite family $\{k'_{i}\}$ of $k'$-algebras
			such that each $k'_{i}$ is a filtered direct limit of \'etale $k'$-algebras
			and $\prod k'_{i}$ is faithfully flat over $k'$.
		\item
			Define the site $\Spec k^{\ind\rat}_{\pro\et}$
			to be the category $k^{\ind\rat}$ with this class of coverings.
		\item
			The cohomology functor for $\Spec k^{\ind\rat}_{\pro\et}$
			at the object $k' \in k^{\ind\rat}$ is denoted by $H^{n}(k'_{\pro\et}, \var)$,
			with derived categorical version $R \Gamma(k'_{\pro\et}, \var)$.
	\end{enumerate}
\end{Def}

An object of $k^{\rat}$ is the perfection of the ring of \emph{rational} functions
on a not-necessarily-irreducible variety over $k$,
whence ``$\rat$''.
These sites are related to $\Spec k^{\perar}_{\et}$ and $\Spec k^{\perf'}_{\pro\fppf}$
by premorphisms of sites:

\begin{Def} \label{def: premorphisms between the four sites}
	Define
		\[
				\Spec k^{\perf'}_{\pro\fppf}
			\overset{\Tilde{h}}{\to}
				\Spec k^{\ind\rat}_{\pro\et}
			\overset{\varepsilon}{\to}
				\Spec k^{\ind\rat}_{\et}
			\overset{\alpha}{\to}
				\Spec k^{\perar}_{\et}
		\]
	to be the premorphisms of sites defined by the inclusion functors on the underlying categories.
	Note that their composite is $\artin{h}$.
\end{Def}

The composite $\varepsilon \compose \Tilde{h}$ was denoted by $h$
in \cite[\S 3.5]{Suz13}.
(To be precise, this reference used
the affine version $\Spec k^{\perf}_{\pro\fppf}$ of $\Spec k^{\perf'}_{\pro\fppf}$.)
All the three premorphisms $h$, $\artin{h}$ and $\Tilde{h}$ are
defined on $\Spec k^{\perf'}_{\pro\fppf}$ (or $\Spec k^{\perf}_{\pro\fppf}$),
and their primary usages are roughly the same:
to pin down pro-algebraic groups by perfect-field-valued points.
The target sites are different, though.

The premorphism $\varepsilon$ is a morphism of sites,
which is the change-of-topology morphism on the category $k^{\ind\rat}$
between pro-\'etale and \'etale,
so that $\varepsilon^{\ast}$ is the pro-\'etale sheafification functor.
The premorphism $\alpha$ is a change-of-category premorphism,
whose pushforward (but not pullback) functor is exact.
(The choice of the letter $\varepsilon$ comes from the fact that
some references use $\varepsilon$ to denote a change-of-topology morphism,
such as \cite[paragraph before Theorem (6.7)]{BK86} between \'etale and Zariski.
On the other hand, there might not exist a common convention
for the choice of letters for change-of-category premorphisms such as the above $\alpha$.)

There are analogues of
Theorem \ref{thm: derived pull of proalg group with finite components}
and Proposition \ref{prop: EIP embeds into sheaves on perar site}
for $\Tilde{h} \colon \Spec k^{\perf'}_{\pro\fppf} \to \Spec k^{\ind\rat}_{\pro\et}$,
basically proved in \cite{Suz13} and \cite{Suz14}:

\begin{Prop} \label{prop: EIP is preserved from profppf to indrat proet}
	Let $\Tilde{\mathcal{E}}_{k}$ be the full subcategory of $\Ab(k^{\perf'}_{\pro\fppf})$
	consisting of extensions of \'etale group schemes by objects of $\Ind \Pro' \Alg / k$.
	For any $G \in \Tilde{\mathcal{E}}_{k}$,
	denote its image by $\Tilde{h}_{\ast}$ by the same symbol $G$.
	\begin{enumerate}
		\item \label{item: profppf to indrat proet, push and pull}
			$G \isomto R \Tilde{h}_{\ast} G$ and $L \Tilde{h}^{\ast} G \isomto G$.
		\item \label{item: profppf to indrat proet, abelian fully faithful}
			$\Tilde{h}_{\ast}$ maps $\Tilde{\mathcal{E}}_{k}$ and $\mathcal{E}_{k}$
			fully faithfully onto their essential images.
		\item \label{item: profppf to indrat proet, triangulated fully faithful}
			$R \Tilde{h}_{\ast}$ gives an equivalence from
			$\genby{\Tilde{\mathcal{E}}_{k}}_{D(k^{\perf'}_{\pro\fppf})}$
			to $\genby{\Tilde{\mathcal{E}}_{k}}_{D(k^{\ind\rat}_{\pro\et})}$
			and an equivalence from
			$\genby{\mathcal{E}_{k}}_{D(k^{\perf'}_{\pro\fppf})}$
			to $\genby{\mathcal{E}_{k}}_{D(k^{\ind\rat}_{\pro\et})}$.
	\end{enumerate}
\end{Prop}

\begin{proof}
	The isomorphism $G \isomto R \Tilde{h}_{\ast} G$ follows from
	\cite[Proposition (2.1.2) (c) and (g)]{Suz14}.
	We show  $L \Tilde{h}^{\ast} G \isomto G$.
	The case $G \in \Pro' \Alg / k$ follows from
	\cite[Proposition 3.7.3 and Remark 3.8.4]{Suz13} (see also \cite[Appendix A]{Suz14}).
	This implies the case $G \in \Ind \Pro' \Alg / k$
	since $L_{n} \Tilde{h}^{\ast}$ commutes with filtered direct limits
	as we saw in the proof of Proposition \ref{prop: Mac Lane calculates derived pull}.
	The case $G$ is an \'etale group scheme can be proven
	similarly to Proposition \ref{prop: derived pull of etale groups}.
	These imply the general case.
\end{proof}

We bring objects from $D(k^{\perf'}_{\pro\fppf})$ to $D(k^{\ind\rat}_{\pro\et})$.

\begin{Def}
	We define
		$
				\artin{\assoc}
			=
				R \Tilde{h}_{\ast} L \artin{h}^{\ast}
			\colon
				D(k^{\perar}_{\et})
			\to
				D(k^{\ind\rat}_{\pro\et})
		$.
\end{Def}

We think of this as an analogue of the pro-\'etale sheafification functor $\varepsilon^{\ast}$.
We will denote the latter functor $\varepsilon^{\ast}$ by $\assoc$
in Definition \ref{def: proetale sheafification} below,
since it is the usual notation \cite[II, D\'efinition 3.5]{AGV72a}
for the sheafification or ``associated sheaf'' functor.
Sheafification commutes with (derived) tensor products.
For $\artin{\assoc}$, we still have a cup product morphism
under an $\artin{h}$-compatibility assumption:

\begin{Prop} \label{prop: modified associated has cup product}
	For any $F, F' \in D(k^{\perar}_{\et})$ such that $F$ is $\artin{h}$-compatible,
	we have a canonical morphism
		\[
				\artin{\assoc}(F) \tensor^{L} \artin{\assoc}(F')
			\to
				\artin{\assoc}(F \tensor^{L} F')
		\]
	in $D(k^{\ind\rat}_{\pro\et})$ functorial in $F$ and $F'$.
\end{Prop}

\begin{proof}
	By Propositions \ref{prop: sheafified functoriality and cup product}
	and \ref{prop: cocup product},
	we have canonical morphisms
		\[
					R \Tilde{h}_{\ast} L \artin{h}^{\ast} F
				\tensor^{L}
					R \Tilde{h}_{\ast} L \artin{h}^{\ast} F'
			\to
				R \Tilde{h}_{\ast}
				(L \artin{h}^{\ast} F \tensor^{L} L \artin{h}^{\ast} F')
			\gets
				R \Tilde{h}_{\ast} L \artin{h}^{\ast} (F \tensor^{L} F'),
		\]
	the latter of which is an isomorphism since $F$ is $\artin{h}$-compatible.
	This gives the desired morphism.
\end{proof}

We bring Proposition \ref{prop: EIP is preserved from profppf to indrat proet}
to $D(k^{\ind\rat}_{\pro\et})$:

\begin{Prop} \label{prop: modified associated preserves EIP} \mbox{}
	\begin{enumerate}
		\item \label{item: modified associated maps EIP to itself}
			The functor $\artin{\assoc} \colon D(k^{\perar}_{\et}) \to D(k^{\ind\rat}_{\pro\et})$
			maps the subcategory $\genby{\mathcal{E}_{k}}_{D(k^{\perar}_{\et})}$
			to the subcategory $\genby{\mathcal{E}_{k}}_{D(k^{\ind\rat}_{\pro\et})}$.
		\item \label{item: modified associated preseves cohomology}
			For any $G \in D^{b}(k^{\perar}_{\et})$ whose cohomologies are in $\mathcal{E}_{k}$,
			we have a canonical isomorphism
				\[
					H^{n}(G) \cong H^{n}(\artin{\assoc}(G))
				\]
			in $\mathcal{E}_{k}$ for all $n$.
	\end{enumerate}
\end{Prop}

\begin{proof}
	This follows from
	Propositions \ref{prop: L h pull preseves groups with finite components}
	and \ref{prop: EIP is preserved from profppf to indrat proet}.
\end{proof}

The presence of $L \artin{h}^{\ast}$ in the definition of $\artin{\assoc}$
makes it difficult to calculate derived sections
$R \Gamma(k'_{\pro\et}, \artin{\assoc}(\var))$ of objects $\artin{\assoc}(\var)$
over each $k' \in k^{\ind\rat}$.
The situation is better under an $\artin{h}$-acyclicity assumption;
see the proof of Proposition \ref{prop: derived section of R bold Gamma} below.
Here are criteria of $\artin{h}$-acyclicity:

\begin{Prop} \label{prop: criteison of h acyclicity} \mbox{}
	\begin{enumerate}
		\item \label{item: h acyclicity is preserved by direct limit}
			If $\{F_{\lambda}\}$ is a filtered direct system in $\Ab(k^{\perar}_{\et})$
			consisting of $\artin{h}$-acyclic objects,
			then its direct limit is $\artin{h}$-acyclic.
		\item \label{item: infinite succesive extension of Ga is h acyclic}
			If $G \in \Pro' \Alg / k$ can be written as
			$\invlim_{n} G_{n}$ with $G_{n} \in \Alg / k$
			such that the transition morphisms $G_{n + 1} \to G_{n}$ are surjective
			with connected unipotent kernel,
			then $G$ as an object of $D(k^{\perar}_{\et})$ is $\artin{h}$-acyclic.
	\end{enumerate}
\end{Prop}

\begin{proof}
	The first statement follows from the fact that
	$L_{n} \artin{h}^{\ast}$ commutes with filtered direct limits
	and $R^{n} \artin{h}_{\ast}$ also commutes with filtered direct limits
	by \cite[Proposition (2.2.4) (b)]{Suz14}.
	For the second statement,
	the assumption implies that $G \in \Pro'_{\fc} \Alg / k$,
	hence $L \artin{h}^{\ast} G \cong G$ by
	Theorem \ref{thm: derived pull of proalg group with finite components}.
	By Proposition \ref{prop: EIP is preserved from profppf to indrat proet}
	(\ref{item: profppf to indrat proet, push and pull}),
	we have $R \Tilde{h}_{\ast} G \cong G$.
	Since $R \artin{h}_{\ast} = \alpha_{\ast} R \varepsilon_{\ast} R \Tilde{h}_{\ast}$,
	we have $R \artin{h}_{\ast} L \artin{h}^{\ast} G \cong \alpha_{\ast} R \varepsilon_{\ast} G$.
	Hence it is enough to show that
	$R \Gamma(k'_{\et}, G) \isomto R \Gamma(k'_{\pro\et}, G)$
	for any perfect field extension $k'$ over $k$.
	This can be proven in the same way as \cite[Proposition (2.4.2) (b)]{Suz14}
	(or is reduced to it).
\end{proof}

%%%%%%%%%%%%%%%%%%%%%%%%%%%%%%%%%%%%%%%%%%%%%%%%%%%%%%%%%%%%%%%%%%%%%%%%%%%%%

\section{A duality formalism for local fields}
\label{sec: A duality formalism for local fields}
Let $K$ be a complete discrete valuation field with perfect residue field $k$ of characteristic $p > 0$.
Denote its ring of integers by $\Order_{K}$ and maximal ideal by $\ideal{p}_{K}$.
If $K$ has mixed characteristic, then $\Order_{K}$ is a finite free $W(k)$-algebra.
If $K$ has equal characteristic, then $\Order_{K}$ is a pro-finite-length $k$-algebra,
and hence a pro-finite-length $W(k)$-algebra via the reduction map $W(k) \onto k$.
As in \cite[\S 2.3]{Suz13}, we make the following definition.

\begin{Def} \label{def: O K an K over k}
	For $k' \in k^{\perar}$, we define
		\begin{gather*}
					\alg{O}_{K}(k')
				=
					W(k') \ctensor_{W(k)} \Order_{K}
				=
					\invlim_{n} \bigl(
						W_{n}(k') \ctensor_{W_{n}(k)} \Order_{K} / \ideal{p}_{K}^{n}
					\bigr),
			\\
					\alg{K}(k')
				=
					\alg{O}_{K}(k') \tensor_{\Order_{K}} K.
		\end{gather*}
\end{Def}

The functors $k' \mapsto \alg{O}_{K}(k')$ and $\alg{K}(k')$ commute with finite products,
taking values in the categories of $\Order_{K}$-algebras and of $K$-algebras, respectively.
If $k'$ has only one direct factor (hence a perfect field extension of $k$),
then $\alg{O}_{K}(k')$ is a complete discrete valuation ring
with maximal ideal $\ideal{p}_{K} \alg{O}_{K}(k')$
and residue field $k'$,
and $\alg{K}(k')$ is its fraction field.

We consider the fppf sites of $\Order_{K}$ and $K$.
To be precise:

\begin{Def} \mbox{}
	\begin{enumerate}
		\item
			Define $\Spec \Order_{K, \fppf}$ (resp.\ $\Spec K_{\fppf}$) to be
			the category of $\Order_{K}$-algebras (resp.\ $K$-algebras) endowed with the fppf topology.
		\item
			The sheaf-Hom functor for $\Spec \Order_{K, \fppf}$ (resp.\ $\Spec K_{\fppf}$)
			is denoted by $\sheafhom_{\Order_{K}}$ (resp.\ $\sheafhom_{K}$).
	\end{enumerate}
\end{Def}

We have the following ``structure morphisms of $\Order_{K}$ and $K$ over $k$'':

\begin{Prop} \mbox{}
	\begin{enumerate}
		\item \label{item: structure morphisms, definition}
			The functors $\alg{O}_{K}$, $\alg{K}$ define premorphisms of sites
				\[
						\artin{\pi}_{\Order_{K}}
					\colon
						\Spec \Order_{K, \fppf}
					\to
						\Spec k^{\perar}_{\et},
					\quad
						\artin{\pi}_{K}
					\colon
						\Spec K_{\fppf}
					\to
						\Spec k^{\perar}_{\et},
				\]
			respectively.
		\item \label{item: structure morphisms, compatibility}
			We have $\artin{\pi}_{K} = \artin{\pi}_{\Order_{K}} \compose j$,
			where $j \colon \Spec K_{\fppf} \into \Spec \Order_{K, \fppf}$
			is the morphism induced by the inclusion $j \colon \Spec K \into \Spec \Order_{K}$.
	\end{enumerate}
\end{Prop}

\begin{proof}
	Coverings in $\Spec k^{\perar}$ are finite extensions of perfect field extensions of $k$
	up to finite products.
	Let $k'' / k'$ be a finite extension of perfect field extensions of $k$.
	Let $f(x)$ be the minimal polynomial of a generator of $k'' / k'$.
	Then $W(k'') \cong W(k')[x] / (\Tilde{f}(x))$ by
	\cite[I, \S 6, Corollaries to Proposition 15; II, \S 5, Theorem 3]{Ser79},
	which is finite free \'etale over $W(k')$,
	where $\Tilde{f}(x)$ is any lift of $f(x)$.
	Taking the completed tensor product with $\Order_{K}$,
	we know that $\alg{O}_{K}(k'') / \alg{O}_{K}(k')$ is a finite free \'etale covering
	and hence an fppf covering.
	Therefore $\alg{O}_{K}$ preserves covering families.
	For any other perfect field extension $k'''$ of $k'$,
	the tensor product $k'' \tensor_{k'} k'''$ is a finite product of finite extensions of $k'''$.
	Hence $W(k'') \tensor_{W(k')} W(k''')$ is isomorphic to $W(k'' \tensor_{k'} k''')$.
	This implies that $\alg{O}_{K}(k'') \tensor_{\alg{O}_{K}(k')} \alg{O}_{K}(k''')$
	is isomorphic to $\alg{O}(k'' \tensor_{k'} k''')$.
	This shows that $\artin{\pi}_{\Order_{K}}$ is a premorphism of sites.
	We have $\artin{\pi}_{K} = \artin{\pi}_{\Order_{K}} \compose j$ obviously.
	Hence $\artin{\pi}_{K}$ is also a premorphism of sites.
\end{proof}

We will define ``cohomology of $\Order_{K}$ and $K$
with an additional structure as a complex of sheaves on $\Spec k^{\ind\rat}_{\pro\et}$''
using $\Spec k^{\perar}_{\et}$.
We will use the very general theorem \cite[Theorem 14.3.1 (vi)]{KS06}
on existence of derived functors in Grothendieck categories.

\begin{Def} \mbox{}
	\begin{enumerate}
		\item
			Define
				\begin{gather*}
								R \artin{\alg{\Gamma}}(\Order_{K}, \var)
							:=
								\artin{\assoc} \compose R \artin{\pi}_{\Order_{K}, \ast}
						\colon
								D(\Order_{K, \fppf})
							\to
								D(k^{\ind\rat}_{\pro\et}),
					\\
								R \artin{\alg{\Gamma}}(K, \var)
							:=
								\artin{\assoc} \compose R \artin{\pi}_{K, \ast}
						\colon
								D(K_{\fppf})
							\to
								D(k^{\ind\rat}_{\pro\et}),
				\end{gather*}
			where $\artin{\pi}_{\Order_{K}, \ast} = (\artin{\pi}_{\Order_{K}})_{\ast}$
			and $\artin{\pi}_{K, \ast} = (\artin{\pi}_{K})_{\ast}$.
		\item
			Define
				\[
						\artin{\pi}_{x, \ast}
					=
						[
								\artin{\pi}_{\Order_{K}, \ast}
							\to
								\artin{\pi}_{K, \ast} j^{\ast}
						][-1]
					\colon
						\Ch(\Order_{K, \fppf})
					\to
						\Ch(k^{\perar}_{\et}),
				\]
			where $[\var]$ denotes the mapping cone.
			We have its right derived functor
				\[
						R \artin{\pi}_{x, \ast}
					\colon
						D(\Order_{K, \fppf})
					\to
						D(k^{\perar}_{\et}),
				\]
			Define
				\[
							R \artin{\alg{\Gamma}}_{x}(\Order_{K}, \var)
						:=
							\artin{\assoc} \compose R \artin{\pi}_{x, \ast}
					\colon
							D(\Order_{K, \fppf})
						\to
							D(k^{\ind\rat}_{\pro\et}).
				\]
		\item
			We denote $\artin{\alg{H}}^{n}(\Order_{K}, \var) = H^{n} R \artin{\alg{\Gamma}}(\Order_{K}, \var)$
			and use the similar notation
			$\artin{\alg{H}}^{n}_{x}(\Order_{K}, \var)$, $\artin{\alg{H}}^{n}(K, \var)$.
	\end{enumerate}
\end{Def}

The subscript $x$ is meant to be the closed subscheme $\Spec k \subset \Spec \Order_{K}$,
so $R \artin{\alg{\Gamma}}_{x}(\Order_{K}, \var)$ is the ``cohomology of $\Spec \Order_{K}$ with support on $x$''.
The restriction functor $j^{\ast}$ as above will be frequently omitted by abuse of notation.
By definition, we have a canonical distinguished triangle
	\begin{equation} \label{eq: localization triangle}
			R \artin{\alg{\Gamma}}_{x}(\Order_{K}, F)
		\to
			R \artin{\alg{\Gamma}}(\Order_{K}, F)
		\to
			R \artin{\alg{\Gamma}}(K, F)
	\end{equation}
in $D(k^{\ind\rat}_{\pro\et})$ functorial in $F \in D(\Order_{K, \fppf})$,
which we call the \emph{localization triangle}.

To understand these cohomology functors,
we need to know their (derived) sections.
Under suitable $\artin{h}$-acyclicity assumptions,
this is given as follows.

\begin{Prop} \label{prop: derived section of R bold Gamma} \mbox{}
	\begin{enumerate}
		\item \label{item: Leray for cohomology of local ring}
			Let $G \in D(\Order_{K, \fppf})$.
			Assume that $R \artin{\pi}_{\Order_{K}, \ast} G$ is $\artin{h}$-acyclic.
			Then there exists a canonical isomorphism
				\[
						R \Gamma \bigl(
							k'_{\pro\et}, R \artin{\alg{\Gamma}}(\Order_{K}, G)
						\bigr)
					\cong
						R \Gamma(\alg{O}_{K}(k')_{\fppf}, G)
				\]
			in $D(\Ab)$ for any $k' \in k^{\perar}$.
			In particular, if $G$ is bounded below, then we have a spectral sequence
				\[
						E_{2}^{i j}
					=
						H^{i} \bigl(
							k'_{\pro\et}, \artin{\alg{H}}^{j}(\Order_{K}, G)
						\bigr)
					\converges
						H^{i + j}(\alg{O}_{K}(k')_{\fppf}, G),
				\]
			and if moreover $k'$ is an algebraically closed field, then we have an isomorphism
				\[
						\artin{\alg{H}}^{n}(\Order_{K}, G)(k')
					\cong
						H^{n}(\alg{O}_{K}(k')_{\fppf}, G)
				\]
			for any $n$.
		\item \label{item: Leray for cohomology of local field}
			A similar statement to (\ref{item: Leray for cohomology of local ring}) holds
			with $\Order_{K}$ and $\alg{O}_{K}$ replaced by $K$ and $\alg{K}$, respectively.
		\item \label{item: Leray for cohomology with support of local ring}
			Let $G \in D(\Order_{K, \fppf})$.
			Assume that $R \artin{\pi}_{\Order_{K}, \ast} G$
			and $R \artin{\pi}_{K, \ast} j^{\ast} G$ are both $\artin{h}$-acyclic.
			Then there exists a canonical isomorphism
				\[
						R \Gamma \bigl(
							k'_{\pro\et}, R \artin{\alg{\Gamma}}_{x}(\Order_{K}, G)
						\bigr)
					\cong
						R \Gamma_{x}(\alg{O}_{K}(k')_{\fppf}, G)
				\]
			in $D(\Ab)$ for any $k' \in k^{\perar}$
			(where the right-hand side is the usual fppf cohomology with support on
			the closed subscheme $\Spec k' \subset \Spec \alg{O}_{K}(k')$).
			In particular, if $G$ is bounded below, then we have a spectral sequence
				\[
						E_{2}^{i j}
					=
						H^{i} \bigl(
							k'_{\pro\et}, \artin{\alg{H}}^{j}_{x}(\Order_{K}, G)
						\bigr)
					\converges
						H_{x}^{i + j}(\alg{O}_{K}(k')_{\fppf}, G),
				\]
			and if moreover $k'$ is an algebraically closed field, then we have an isomorphism
				\[
						\artin{\alg{H}}^{n}_{x}(\Order_{K}, G)(k')
					\cong
						H^{n}_{x}(\alg{O}_{K}(k')_{\fppf}, G)
				\]
			for any $n$.
	\end{enumerate}
\end{Prop}

\begin{proof}
	For (\ref{item: Leray for cohomology of local ring}),
	since $R \artin{\pi}_{\Order_{K}, \ast} G$ is $\artin{h}$-acyclic, we have
		\[
				R \artin{\pi}_{\Order_{K}, \ast} G
			\isomto
				R \artin{h}_{\ast} L \artin{h}^{\ast} R \artin{\pi}_{\Order_{K}, \ast} G
			\cong
				\alpha_{\ast} R \varepsilon_{\ast}
				\artin{\assoc}(R \artin{\pi}_{\Order_{K}, \ast} G)
			=
				\alpha_{\ast} R \varepsilon_{\ast}
				R \artin{\alg{\Gamma}}(\Order_{K}, G).
		\]
	Taking $R \Gamma(k'_{\et}, \var)$, we get the result,
	noting that an algebraically closed field is w-contractible \cite[Definition 2.4.1]{BS15}
	(see also the proof of \cite[Proposition 2.5.2]{Suz18a}).
	Assertion (\ref{item: Leray for cohomology of local field}) can be proven similarly.
	Assertion (\ref{item: Leray for cohomology with support of local ring}) follows from
	(\ref{item: Leray for cohomology of local ring}) and
	(\ref{item: Leray for cohomology of local field}).
\end{proof}

These cohomology functors support a cup product formalism:

\begin{Prop} \label{prop: cup product for local fields} \mbox{}
	\begin{enumerate}
		\item \label{item: cup product for K}
			There exists a canonical morphism
				\[
						R \artin{\alg{\Gamma}}(K, G)
					\tensor^{L}
						R \artin{\alg{\Gamma}}(K, G')
					\to
						R \artin{\alg{\Gamma}}(K, G \tensor^{L} G')
				\]
			in $D(k^{\ind\rat}_{\pro\et})$ functorial in $G, G' \in D(K_{\fppf})$ such that
			$R \artin{\pi}_{K, \ast} G$ is $\artin{h}$-compatible.
		\item \label{item: cup product for O K}
			There exists a canonical morphism
				\[
						R \artin{\alg{\Gamma}}(\Order_{K}, G)
					\tensor^{L}
						R \artin{\alg{\Gamma}}_{x}(\Order_{K}, G')
					\to
						R \artin{\alg{\Gamma}}_{x}(\Order_{K}, G \tensor^{L} G')
				\]
			in $D(k^{\ind\rat}_{\pro\et})$ functorial in $G, G' \in D(\Order_{K, \fppf})$ such that
			$R \artin{\pi}_{\Order_{K}, \ast} G$ is $\artin{h}$-compatible.
	\end{enumerate}
\end{Prop}

\begin{proof}
	(\ref{item: cup product for K})
	This follows from
	Proposition \ref{prop: sheafified functoriality and cup product} applied to $\artin{\pi}_{K}$
	and Proposition \ref{prop: modified associated has cup product}.
	
	(\ref{item: cup product for O K})
	It is enough to construct a canonical morphism
		\[
					R \artin{\pi}_{\Order_{K}, \ast} G
				\tensor^{L}
					R \artin{\pi}_{x, \ast} G'
			\to
				R \artin{\pi}_{x, \ast}(G \tensor^{L} G')
		\]
	in $D(k^{\perar}_{\et})$ functorial in (arbitrary) $G, G' \in D(\Order_{K, \fppf})$.
	By the same method as the construction of the morphism \eqref{eq: cup product} of
	Proposition \ref{prop: sheafified functoriality and cup product},
	it is enough to construct a canonical morphism
		\[
				R \artin{\pi}_{x, \ast}
				R \sheafhom_{\Order_{K}}(G, G'')
			\to
				R \sheafhom_{k^{\perar}_{\et}} \bigl(
					R \artin{\pi}_{\Order_{K}, \ast} G,
					R \artin{\pi}_{x, \ast} G''
				\bigr)
		\]
	in $D(k^{\perar}_{\et})$ functorial in $G, G'' \in D(\Order_{K, \fppf})$.
	By the same method as the construction of the morphism \eqref{eq: derived push and RHom} of
	Proposition \ref{prop: sheafified functoriality and cup product},
	it is enough to construct a canonical morphism
		\[
				\artin{\pi}_{x, \ast}
				\sheafhom_{\Order_{K}}(G, G'')
			\to
				\sheafhom_{k^{\perar}_{\et}} \bigl(
					\artin{\pi}_{\Order_{K}, \ast} G,
					\artin{\pi}_{x, \ast} G''
				\bigr)
		\]
	in $\Ch(k^{\perar}_{\et})$ functorial in $G, G'' \in \Ch(\Order_{K, \fppf})$.
	The construction is given by applying the functoriality of mapping fibers
	to the commutative diagram
		\[
			\begin{CD}
					\artin{\pi}_{\Order_{K}, \ast}
					\sheafhom_{\Order_{K}}(G, G'')
				@>>>
					\sheafhom_{k^{\perar}_{\et}} \bigl(
						\artin{\pi}_{\Order_{K}, \ast} G,
						\artin{\pi}_{\Order_{K}, \ast} G''
					\bigr)
				\\
				@VVV
				@VVV
				\\
					\artin{\pi}_{K, \ast} j^{\ast}
					\sheafhom_{\Order_{K}}(G, G'')
				@>>>
					\sheafhom_{k^{\perar}_{\et}} \bigl(
						\artin{\pi}_{\Order_{K}, \ast} G,
						\artin{\pi}_{K, \ast} j^{\ast} G''
					\bigr)
			\end{CD}
		\]
	in $\Ch(k^{\perar}_{\et})$.
\end{proof}

The next proposition shows how the above cup product morphisms
for $\Order_{K}$ and $K$ are compatible to each other.
It is a version of \cite[Proposition (3.3.7)]{Suz14} for $R \artin{\alg{\Gamma}}$.
This type of compatibility is important in applications
in order to deduce a duality result for $K$
from that of $\Order_{K}$ (\cite[Proposition (5.2.2.2)]{Suz14} for example)
and, conversely in some cases, a duality result for $\Order_{K}$ from that of $K$
(\cite[Proposition 2.5.4]{Suz18a} for example).

\begin{Prop} \label{prop: cup product and localization}
	Let $G, F \in D(\Order_{K, \fppf})$.
	To simplify the notation,
	we denote
		\[
					[\var, \var]_{\Order_{K}}
				=
					R \sheafhom_{\Order_{K}},
			\quad
					[\var, \var]_{K}
				=
					R \sheafhom_{K},
			\quad
					[\var, \var]_{k}
				=
					R \sheafhom_{k^{\ind\rat}_{\pro\et}},
		\]
		\[
					R \artin{\alg{\Gamma}}_{x}
				=
					R \artin{\alg{\Gamma}}_{x}(\Order_{K}, \var),
			\quad
					R \artin{\alg{\Gamma}}_{\Order_{K}}
				=
					R \artin{\alg{\Gamma}}(\Order_{K}, \var),
			\quad
					R \artin{\alg{\Gamma}}_{K}
				=
					R \artin{\alg{\Gamma}}(K, \var).
		\]
	Then we have a morphism of distinguished triangles
		\[
			\begin{CD}
					R \artin{\alg{\Gamma}}_{x} [G, F]_{\Order_{K}}
				@>>>
					R \artin{\alg{\Gamma}}_{\Order_{K}} [G, F]_{\Order_{K}}
				@>>>
					R \artin{\alg{\Gamma}}_{K} [G, F]_{K}
				\\
				@VVV
				@VVV
				@VVV
				\\
					[R \artin{\alg{\Gamma}}_{\Order_{K}} G,
					R \artin{\alg{\Gamma}}_{x} F]_{k}
				@>>>
					[R \artin{\alg{\Gamma}}_{x} G,
					R \artin{\alg{\Gamma}}_{x} F]_{k}
				@>>>
					[R \artin{\alg{\Gamma}}_{K} G,
					R \artin{\alg{\Gamma}}_{x} F]_{k}[1]
			\end{CD}
		\]
	in $D(k^{\ind\rat}_{\pro\et})$,
	where the horizontal triangles are the localization triangles
	\eqref{eq: localization triangle},
	the left two vertical morphisms are
	the morphism in Proposition \ref{prop: cup product for local fields} (\ref{item: cup product for O K})
	translated by the derived tensor-Hom adjunction,
	and the right vertical morphism is
	the morphism in Proposition \ref{prop: cup product for local fields} (\ref{item: cup product for K})
	translated similarly
	composed with the connecting morphism
	$R \artin{\alg{\Gamma}}_{K} F \to R \artin{\alg{\Gamma}}_{x} F[1]$
	of the localization triangle.
\end{Prop}

Note that there is a hidden square next to the right square in the diagram
since we are hiding the shifted terms of distinguished triangles from the notation.

\begin{proof}
	Denote the total complex of the sheaf-Hom double complex functor $\sheafhom_{\Order_{K}}$
	by $[\var, \var]_{\Order_{K}}^{c}$.
	Use the notation $[\var, \var]_{K}^{c}$ similarly.
	Denote the total complex of the sheaf-Hom double complex functor $\sheafhom_{k^{\perar}_{\et}}$
	for $\Spec k^{\perar}_{\et}$ by $[\var, \var]_{k}'^{c}$.
	Let $G \isomto I$ and $F \isomto J$ be quasi-isomorphisms to K-injective complexes.
	We can check that the natural diagram
		\[
			\begin{CD}
					\artin{\pi}_{x, \ast} [I, J]_{\Order_{K}}^{c}
				@>>>
					\artin{\pi}_{\Order_{K}, \ast} [I, J]_{\Order_{K}}^{c}
				@>>>
					\artin{\pi}_{K, \ast} [I, J]_{K}^{c}
				\\
				@VVV
				@VVV
				@VVV
				\\
					[\artin{\pi}_{\Order_{K}, \ast} I,
					\artin{\pi}_{x, \ast} J]_{k}'^{c}
				@>>>
					[\artin{\pi}_{x, \ast} I,
					\artin{\pi}_{x, \ast} J]_{k}'^{c}
				@>>>
					[\artin{\pi}_{K, \ast} I,
					\artin{\pi}_{x, \ast} J]_{k}'^{c}[1]
			\end{CD}
		\]
	in $\Ch(k^{\perar}_{\et})$ is commutative up to homotopy
	(where again there is a hidden square next to the right one).
	Applying the localization morphism $[\var, \var]_{k}'^{c} \to [\var, \var]_{k}'$
	(where $[\var, \var]_{k}' = R \sheafhom_{k^{\perar}_{\et}}$)
	to the lower triangle,
	we have a morphism of distinguished triangles
		\[
			\begin{CD}
					R \artin{\pi}_{x, \ast} [G, F]_{\Order_{K}}
				@>>>
					R \artin{\pi}_{\Order_{K}, \ast} [G, F]_{\Order_{K}}
				@>>>
					R \artin{\pi}_{K, \ast} [G, F]_{K}
				\\
				@VVV
				@VVV
				@VVV
				\\
					[R \artin{\pi}_{\Order_{K}, \ast} G,
					R \artin{\pi}_{x, \ast} F]_{k}'
				@>>>
					[R \artin{\pi}_{x, \ast} G,
					R \artin{\pi}_{x, \ast} F]_{k}'
				@>>>
					[R \artin{\pi}_{K, \ast} G,
					R \artin{\pi}_{x, \ast} F]_{k}'[1]
			\end{CD}
		\]
	in $D(k^{\perar}_{\et})$.
	Applying $\artin{\assoc}$ and using the morphism
		\[
				\artin{\assoc}([G', F']_{k}')
			\to
				[\artin{\assoc}(G'), \artin{\assoc}(F')]_{k}
		\]
	for $G', F' \in D(k^{\perar}_{\et})$ coming from
	Proposition \ref{prop: modified associated has cup product},
	we get the result.
\end{proof}

%%%%%%%%%%%%%%%%%%%%%%%%%%%%%%%%%%%%%%%%%%%%%%%%%%%%%%%%%%%%%%%%%%%%%%%%%%%%%

\section{Trace morphisms and a finiteness property of cohomology}
\label{sec: Trace morphisms and a finiteness property of cohomology}

In this section, we prove two statements
that are keys in order to apply the formalism in the previous section and obtain duality results.
The corresponding statements in \cite{Suz13} and \cite{Suz14} in the older formalism are proved
using some exotic approximation arguments.
The proofs in this section are self-contained and much more standard.

The first statement is the existence of a trace (iso)morphism in this formalism.
In the older formalism, it is \cite[Proposition 2.4.4]{Suz13} and \cite[(5.2.1.1)]{Suz14}.

\begin{Prop} \label{prop: trace morphism}
	There exists a canonical isomorphism
		\[
				R \artin{\alg{\Gamma}}_{x}(\Order_{K}, \Gm)
			\cong
				\Z[-1],
		\]
	which we call the \emph{trace isomorphism}.
	The composite
		\[
				R \artin{\alg{\Gamma}}(K, \Gm)
			\to
				R \artin{\alg{\Gamma}}_{x}(\Order_{K}, \Gm)[1]
			\cong
				\Z
		\]
	is called the \emph{trace morphism}.
\end{Prop}

\begin{proof}
	We have $\artin{\pi}_{K, \ast} \Gm = \alg{K}^{\times}$ in $\Ab(k^{\perar}_{\et})$.
	For any perfect field extension $k'$ of $k$,
	the normalized valuation for the discrete valuation field $\alg{K}(k')$
	defines a split surjection $\alg{K}(k')^{\times} \onto \Z$
	functorial in $k'$.
	This uniquely extends to a split surjection $\alg{K}(k')^{\times} \onto \Z(k')$
	functorial in arbitrary $k' \in k^{\perar}$ that commutes with finite products.
	Hence we obtain a split surjection $\alg{K}^{\times} \onto \Z$ in $\Ab(k^{\perar}_{\et})$.
	Its kernel is $\artin{\pi}_{\Order_{K}, \ast} \Gm = \alg{O}_{K}^{\times}$.
	
	For $n \ge 1$, let $1 + \alg{p}_{K}^{n} \subset \alg{O}_{K}^{\times}$ be the subsheaf that assigns
	$k' \mapsto 1 + \ideal{p}_{K}^{n} \alg{O}_{K}(k')$.
	Then $\alg{O}_{K}^{\times} / (1 + \alg{p}_{K}^{1}) \cong \Gm$
	and $(1 + \alg{p}_{K}^{n}) / (1 + \alg{p}_{K}^{n + 1}) \cong \Ga$,
	and $\alg{O}_{K}^{\times}$ is the inverse limit of
	$\alg{O}_{K}^{\times} / (1 + \alg{p}_{K}^{n})$ over $n \ge 1$.
	Hence $\alg{O}_{K}^{\times} \in \Pro'_{\fc} \Alg / k$ and
	it satisfies the condition of
	Proposition \ref{prop: criteison of h acyclicity}
	(\ref{item: infinite succesive extension of Ga is h acyclic}).
	Hence $\alg{O}_{K}^{\times}$ is $\artin{h}$-acyclic.
	So is $\alg{K}^{\times} \cong \alg{O}_{K}^{\times} \times \Z \in \mathcal{E}_{k}$.
	
	We show that $R^{n} \artin{\pi}_{\Order_{K}, \ast} \Gm$ and $R^{n} \artin{\pi}_{K, \ast} \Gm$ are zero
	for any $n \ge 1$.
	They are \'etale sheafifications of the presheaves
		\[
				k' \in k^{\perar}
			\mapsto
				H^{n}(\alg{O}_{K}(k'), \Gm),
				H^{n}(\alg{K}(k'), \Gm).
		\]
	Since $\alg{O}_{K}(k')$ is a finite product of complete discrete valuation rings,
	we have $H^{n}(\alg{O}_{K}(k'), \Gm) \isomto H^{n}(k', \Gm)$,
	which sheafifies to zero.
	To show that the second presheaf sheafifies to zero,
	it is enough to show that for any perfect field extension $k'$ over $k$,
	we have
		\[
				\dirlim_{k'' / k'}
					H^{n}(\alg{K}(k''), \Gm)
			=
				0,
		\]
	where the direct limit is over finite extensions of $k'$ in a fixed algebraic closure of $k'$.
	The direct limit of $\alg{K}(k'')$ over $k'' / k'$ is
	the maximal unramified extension $\alg{K}(k')^{\ur}$ of $\alg{K}(k')$,
	which is an excellent henselian discrete valuation field with algebraically closed residue field.
	Since the direct limit commutes with cohomology,
	the left-hand side is isomorphic to $H^{n}(\alg{K}(k')^{\ur}, \Gm)$.
	The vanishing of this cohomology is classical
	(\cite[Chapter V, Section 4, Proposition 7 and Chapter X, Section 7, Proposition 11]{Ser79}).
	
	Therefore $R \artin{\pi}_{\Order_{K}, \ast} \Gm \cong \alg{O}_{K}^{\times}$ and
	$R \artin{\pi}_{K, \ast} \Gm \cong \alg{K}^{\times}$.
	We apply $\artin{\assoc}$ to them.
	By Proposition \ref{prop: modified associated preserves EIP}
	(\ref{item: modified associated preseves cohomology}),
	we have
		\[
					R \alg{\Gamma}(\Order_{K}, \Gm)
				\cong
					\artin{\assoc}(\alg{O}^{\times})
				\cong
					\alg{O}^{\times},
			\quad
					R \alg{\Gamma}(K, \Gm)
				\cong
					\artin{\assoc}(\alg{K}^{\times})
				\cong
					\alg{K}^{\times}
		\]
	and hence $R \alg{\Gamma}_{x}(\Order_{K}, \Gm) \cong \Z[-1]$.
\end{proof}

The next one states that $R^{n} \artin{\pi}_{K, \ast} G$
is locally of finite presentation for $n \ge 1$
whenever it is representable and $G$ is a smooth group scheme over $K$.
In the older formulation, it is \cite[Proposition (3.4.3) (a)]{Suz14}.
As in \cite[Proposition (3.4.3) (d)]{Suz14},
this is a key step to prove that $\artin{\alg{H}}^{1}(K, A) \in \Ind \Alg / k$
(without a pro-algebraic part) for an abelian variety $A$ over $K$,
though we do not explain the proof of this fact in this paper.

\begin{Prop}
	Let $G$ be a smooth group scheme over $K$ and $n \ge 1$.
	Then $R^{n} \artin{\pi}_{K, \ast} G$ is torsion
	and commutes with filtered direct limits as a functor
	$k^{\perar} \to \Ab$.
\end{Prop}

\begin{proof}
	The sheaf $R^{n} \artin{\pi}_{K, \ast} G$ is
	the \'etale sheafification of the presheaf
		\[
				k' \in k^{\perar}
			\mapsto
				H^{n}(\alg{K}(k'), G).
		\]
	This is torsion since $\alg{K}(k')$ is a finite direct product of fields
	and Galois cohomology in positive degrees is torsion.
	It is enough to show that
		\[
				\dirlim_{\lambda}
					H^{n}(\alg{K}(k_{\lambda}), G)
			\isomto
				H^{n}(\alg{K}(k'), G)
		\]
	for any $k' \in k^{\perar}$ that can be written as a direct limit
	of a filtered direct system $\{k_{\lambda}\}$ in $k^{\perar}$.
	We may assume that the $k_{\lambda}$ and $k'$ are fields.
	The ring $\dirlim_{\lambda} \alg{K}(k_{\lambda})$ is an (excellent) henselian discrete valuation field
	with completion $\alg{K}(k')$.
	Hence they have isomorphic cohomology in positive degrees with coefficients in a smooth group scheme
	by \cite[Proposition 2.5.3 (2) (3)]{GGMB14}.
	This gives the result.
\end{proof}

%%%%%%%%%%%%%%%%%%%%%%%%%%%%%%%%%%%%%%%%%%%%%%%%%%%%%%%%%%%%%%%%%%%%%%%%%%%%%

\section{Comparison with the older formulation}

Recall the morphism of sites
$\varepsilon \colon \Spec k^{\ind\rat}_{\pro\et} \to \Spec k^{\ind\rat}_{\et}$
from Definition \ref{def: premorphisms between the four sites}.

\begin{Def} \label{def: proetale sheafification}
	Define $\assoc = \varepsilon^{\ast}$
	(as $\Ab(k^{\ind\rat}_{\et}) \to \Ab(k^{\ind\rat}_{\pro\et})$
	or $D(k^{\ind\rat}_{\et}) \to D(k^{\ind\rat}_{\pro\et})$),
	which is the pro-\'etale sheafification functor.
\end{Def}

We compare $\artin{\assoc}$ and $\assoc$ applied to objects of $\mathcal{E}_{k}$.

\begin{Prop} \label{prop: comparison of two formulations for groups with finite components}
	For any $G \in \genby{\mathcal{E}_{k}}_{D(k^{\ind\rat}_{\et})}$,
	there exists a canonical isomorphism
		\begin{equation} \label{eq: two associated constructions agree on groups with finite components}
				\artin{\assoc}(\alpha_{\ast} G)
			\cong
				\assoc(G)
		\end{equation}
	in $\genby{\mathcal{E}_{k}}_{k^{\ind\rat}_{\pro\et}}$.
	More precisely, the morphism
		\begin{equation} \label{eq: morphism between the two pulls to pro fppf}
				L \artin{h}^{\ast}
				\alpha_{\ast} G
			\cong
				L \Tilde{h}^{\ast}
				\varepsilon^{\ast}
				L \alpha^{\ast}
				\alpha_{\ast} G
			\to
				L \Tilde{h}^{\ast}
				\varepsilon^{\ast} G
		\end{equation}
	defined by the counit for $\alpha$ is an isomorphism,
	the morphism
		\begin{equation} \label{eq: pull to pro fppf and then push back to ind rat}
				\varepsilon^{\ast} G
			\to
				R \Tilde{h}_{\ast}
				L \Tilde{h}^{\ast}
				\varepsilon^{\ast} G
		\end{equation}
	defined by the unit for $\Tilde{h}$ is an isomorphism,
	and the isomorphism \eqref{eq: two associated constructions agree on groups with finite components}
	is obtained by applying $R \Tilde{h}_{\ast}$ to \eqref{eq: morphism between the two pulls to pro fppf}
	and using \eqref{eq: pull to pro fppf and then push back to ind rat} on the right-hand side.
\end{Prop}

\begin{proof}
	We may assume that $G \in \mathcal{E}_{k}$.
	Since $G$ is a sheaf for the pro-fppf topology and hence for any coarser topology,
	the morphism \eqref{eq: morphism between the two pulls to pro fppf} is of the form
	$L \artin{h}^{\ast} G \to L \Tilde{h}^{\ast} G$.
	But we have $L \artin{h}^{\ast} G \isomto G$ by
	Theorem \ref{thm: derived pull of proalg group with finite components}
	and $L \Tilde{h}^{\ast} G \isomto G$ by
	Proposition \ref{prop: EIP is preserved from profppf to indrat proet}.
	Therefore \eqref{eq: morphism between the two pulls to pro fppf} is an isomorphism.
	The same proposition shows that
	\eqref{eq: pull to pro fppf and then push back to ind rat} is an isomorphism.
\end{proof}

Recall the following definition from \cite[Section 2.4]{Suz14}.

\begin{Def} \mbox{}
	\begin{enumerate}
		\item
			A sheaf $F \in \Ab(k^{\ind\rat}_{\et})$ is said to be \emph{P-acyclic}
			if $F \isomto R \varepsilon_{\ast} \varepsilon^{\ast} F$.
		\item
			An object $F \in D^{+}(k^{\ind\rat}_{\et})$ is said to be \emph{P-acyclic}
			if each cohomology object of $F$ is P-acyclic.
			This implies that $F \isomto R \varepsilon_{\ast} \varepsilon^{\ast} F$.
	\end{enumerate}
\end{Def}

The letter ``'P' means ``pro''; see \cite[Footnote 8]{Suz14} for more details.
Here is the relation to $\artin{h}$-acyclicity:

\begin{Prop} \label{prop: P acyclic implies derived h acyclic}
	If $G \in \genby{\mathcal{E}_{k}}_{k^{\ind\rat}_{\et}}$ is P-acyclic,
	then $\alpha_{\ast} G$ is $\artin{h}$-acyclic.
\end{Prop}

\begin{proof}
	By Proposition \ref{prop: comparison of two formulations for groups with finite components}, we have
		\begin{align*}
		&		R \artin{h}_{\ast} L \artin{h}^{\ast} \alpha_{\ast} G
			\cong
				\alpha_{\ast} R \varepsilon_{\ast} R \Tilde{h}^{\ast} L \artin{h}^{\ast} \alpha_{\ast} G
			\cong
				\alpha_{\ast} R \varepsilon_{\ast} \artin{\assoc}(\alpha_{\ast} G)
		\\
		& \quad
			\cong
				\alpha_{\ast} R \varepsilon_{\ast} \varepsilon^{\ast} G
			\cong
				\alpha_{\ast} G.
		\end{align*}
\end{proof}

We can compare cup products for $\artin{\assoc}$ and $\assoc$ on $\mathcal{E}_{k}$:

\begin{Prop} \label{prop: compatibility the associated constructions}
	Let $\varphi \colon G \tensor^{L} G' \to G''$ be a morphism in $D(k^{\ind\rat}_{\et})$ with
	$G, G', G'' \in \genby{\mathcal{E}_{k}}_{k^{\ind\rat}_{\et}}$.
	Consider the composite of the morphisms
		\[
				\artin{\assoc}(\alpha_{\ast} G) \tensor^{L} \artin{\assoc}(\alpha_{\ast} G')
			\to
				\artin{\assoc}(\alpha_{\ast} G \tensor^{L} \alpha_{\ast} G')
			\to
				\artin{\assoc} \bigl(
					\alpha_{\ast}(G \tensor^{L} G')
				\bigr)
			\to
				\artin{\assoc}(\alpha_{\ast} G''),
		\]
	where the first morphism is given by
	Propositions \ref{prop: modified associated has cup product}
	and \ref{prop: EIP objects are h compatible},
	the second by Proposition \ref{prop: sheafified functoriality and cup product}
	and the third by $\varphi$.
	Also consider the composite of the morphisms
		\[
				\assoc(G) \tensor^{L} \assoc(G')
			\cong
				\assoc(G \tensor^{L} G')
			\to
				\assoc(G''),
		\]
	where the first isomorphism is the obvious isomorphism about sheafification
	and the second $\varphi$.
	These two composite morphisms are compatible under the isomorphism
	\eqref{eq: two associated constructions agree on groups with finite components}.
\end{Prop}

\begin{proof}
	Arguing similarly to the proof of Proposition \ref{prop: EIP objects are h compatible}
	using Proposition \ref{prop: EIP is preserved from profppf to indrat proet}
	instead of Theorem \ref{thm: derived pull of proalg group with finite components},
	we know that $\assoc(G)$ is $\Tilde{h}$-compatible.
	The same is true for $\assoc(G')$ and $\assoc(G'')$.
	Hence we have a canonical isomorphism
		\[
				L \Tilde{h}^{\ast} \bigl(
					\assoc(G) \tensor^{L} \assoc(G')
				\bigr)
			\isomto
				L \Tilde{h}^{\ast} \assoc(G) \tensor^{L} L \Tilde{h}^{\ast} \assoc(G')
		\]
	by Proposition \ref{prop: cocup product}.
	Therefore we have a composite morphism
		\begin{align*}
						R \Tilde{h}_{\ast} L \Tilde{h}^{\ast} \assoc(G)
					\tensor^{L}
						R \Tilde{h}_{\ast} L \Tilde{h}^{\ast} \assoc(G')
			&	\to
					R \Tilde{h}_{\ast} \bigl(
						L \Tilde{h}^{\ast} \assoc(G) \tensor^{L} L \Tilde{h}^{\ast} \assoc(G')
					\bigr)
			\\
			&	\isomfrom
					R \Tilde{h}_{\ast} L \Tilde{h}^{\ast} \assoc(G \tensor^{L} G')
			\\
			&	\to
					R \Tilde{h}_{\ast} L \Tilde{h}^{\ast} \assoc(G'').
		\end{align*}
	Now one checks that the diagram
		\[
			\begin{CD}
					\artin{\assoc}(\alpha_{\ast} G) \tensor^{L} \artin{\assoc}(\alpha_{\ast} G')
				@>>>
					\artin{\assoc}(\alpha_{\ast} G'')
				\\
				@VV \wr V
				@V \wr VV
				\\
						R \Tilde{h}_{\ast} L \Tilde{h}^{\ast} \assoc(G)
					\tensor^{L}
						R \Tilde{h}_{\ast} L \Tilde{h}^{\ast} \assoc(G')
				@>>>
					R \Tilde{h}_{\ast} L \Tilde{h}^{\ast} \assoc(G'')
				\\
				@AA \wr A
				@A \wr AA
				\\
					\assoc(G) \tensor^{L} \assoc(G')
				@>>>
					\assoc(G'')
			\end{CD}
		\]
	is commutative.
	(Be careful that both the upper and middle horizontal morphisms are actually defined as
	zigzags of the form $\bullet \to \bullet \isomfrom \bullet \to \bullet$.)
	This gives the result.
\end{proof}

We recall some of the constructions in \cite[Section 2.5]{Suz18a}.

\begin{Def} \mbox{}
	\begin{enumerate}
		\item
			Define $\alg{O}_{K}(k')$ and $\alg{K}(k')$ for $k' \in k^{\ind\rat}$
			by the same formulas as Definition \ref{def: O K an K over k}.
		\item
			Define premorphisms of sites
				\[
						\pi_{\Order_{K}}
					\colon
						\Spec \Order_{K, \fppf}
					\to
						\Spec k^{\ind\rat}_{\et},
					\quad
						\pi_{K}
					\colon
						\Spec K_{\fppf}
					\to
						\Spec k^{\ind\rat}_{\et}
				\]
			by the functors $\alg{O}_{K}$, $\alg{K}$, respectively
			(which are indeed premorphisms by \cite[Proposition 2.5.1]{Suz18a}).
		\item
			Define
				\begin{gather*}
							\alg{\Gamma}(\Order_{K}, \var)
						=
							\assoc \compose \pi_{\Order_{K}, \ast}
						\colon
							\Ab(\Order_{K, \fppf})
						\to
							\Ab(k^{\ind\rat}_{\pro\et}),
					\\
							\alg{\Gamma}(K, \var)
						=
							\assoc \compose \pi_{K, \ast}
						\colon
							\Ab(K_{\fppf})
						\to
							\Ab(k^{\ind\rat}_{\pro\et}),
				\end{gather*}
			where $\pi_{\Order_{K}, \ast} = (\pi_{\Order_{K}})_{\ast}$
			and $\pi_{K, \ast} = (\pi_{K})_{\ast}$.
			They naturally extend to the categories of complexes.
			Define
				\[
						\alg{\Gamma}_{x}(\Order_{K}, \var)
					=
						\bigl[
								\alg{\Gamma}(\Order_{K}, \var)
							\to
								\alg{\Gamma}(K, \var)
						\bigr][-1]
					\colon
							\Ch(\Order_{K, \fppf})
						\to
							\Ch(k^{\ind\rat}_{\pro\et}).
				\]
		\item
			We have their right derived functors
				\begin{gather*}
							R \alg{\Gamma}(\Order_{K}, \var),
							R \alg{\Gamma}_{x}(\Order_{K}, \var)
						\colon
								D(\Order_{K, \fppf})
							\to
								D(k^{\ind\rat}_{\pro\et}),
					\\
							R \alg{\Gamma}(K, \var)
						\colon
								D(K_{\fppf})
							\to
								D(k^{\ind\rat}_{\pro\et}).
				\end{gather*}
		\item
			We denote $\alg{H}^{n}(\Order_{K}, \var) = H^{n} R \alg{\Gamma}(\Order_{K}, \var)$
			and use the similar notation
			$\alg{H}^{n}_{x}(\Order_{K}, \var)$, $\alg{H}^{n}(K, \var)$.
	\end{enumerate}
\end{Def}

In \cite[Section 3.3]{Suz14},
the functor $\pi_{\Order_{K}, \ast}$ was denoted by $\alg{\Gamma}(\Order_{K}, \var)$
and the functor $\alg{\Gamma}(\Order_{K}, \var)$ was
denoted by $\Tilde{\alg{\Gamma}}(\Order_{K}, \var)$.
Similar for their derived versions and for $K$ instead of $\Order_{K}$.

The relation between these $\pi_{\Order_{K}}, \pi_{K}$
and the previous $\artin{\pi}_{\Order_{K}}, \artin{\pi}_{K}$ is the following.

\begin{Prop} \label{prop: comparison of structure morphisms}
	The composite of
		\[
				\Spec \Order_{K, \fppf}
			\overset{\pi_{\Order_{K}}}{\to}
				\Spec k^{\ind\rat}_{\et}
			\overset{\alpha}{\to}
				\Spec k^{\perar}_{\et}
		\]
	is $\artin{\pi}_{\Order_{K}}$.
	The same relation holds with $\Order_{K}$ replaced by $K$.
\end{Prop}

\begin{proof}
	Obvious.
\end{proof}

We compare $R \alg{\Gamma}$ and $R \artin{\alg{\Gamma}}$.

\begin{Prop} \label{prop: compatibility of cohomology functors} \mbox{}
	\begin{enumerate}
		\item \label{item: comparison of cohomology for O K}
			Let $G \in D(\Order_{K, \fppf})$.
			Assume that
			$R \pi_{\Order_{K}, \ast} G \in \genby{\mathcal{E}_{k}}_{D(k^{\ind\rat}_{\et})}$.
			Then there exists a canonical isomorphism
			$R \artin{\alg{\Gamma}}(\Order_{K}, G) \cong R \alg{\Gamma}(\Order_{K}, G)$
			in $\genby{\mathcal{E}_{k}}_{k^{\ind\rat}_{\pro\et}}$.
		\item \label{item: comparison of cohomology for K}
			Let $G \in D(K_{\fppf})$.
			Assume that
			$R \pi_{K, \ast} G \in \genby{\mathcal{E}_{k}}_{D(k^{\ind\rat}_{\et})}$.
			Then there exists a canonical isomorphism
			$R \artin{\alg{\Gamma}}(K, G) \cong R \alg{\Gamma}(K, G)$
			in $\genby{\mathcal{E}_{k}}_{D(k^{\ind\rat}_{\pro\et})}$.
		\item \label{item: comparison of cohomology with support for O K}
			Let $G \in D(\Order_{K, \fppf})$.
			Assume that $G$ satisfies the assumption of (\ref{item: comparison of cohomology for O K})
			and that $j^{\ast} G$ satisfies the assumption of
			(\ref{item: comparison of cohomology for K}).
			Then there exists a canonical isomorphism
			$R \artin{\alg{\Gamma}}_{x}(\Order_{K}, G) \cong R \alg{\Gamma}_{x}(\Order_{K}, G)$
			in $\genby{\mathcal{E}_{k}}_{D(k^{\ind\rat}_{\pro\et})}$.
	\end{enumerate}
\end{Prop}

\begin{proof}
	This follows from
	Propositions \ref{prop: comparison of two formulations for groups with finite components}
	and \ref{prop: comparison of structure morphisms}.
\end{proof}

The sheaves $R \alg{\Gamma}(\Order_{K}, G)$, $R \alg{\Gamma}(K, G)$,
$R \artin{\alg{\Gamma}}(\Order_{K}, G)$, $R \artin{\alg{\Gamma}}(K, G)$
for most of the groups of interest $G$
satisfy appropriate acyclicity properties
and have cohomologies in $\mathcal{E}_{k}$
by the following proposition:

\begin{Prop} \label{prop: groups of interest satisfy the required conditions} \mbox{}
	\begin{enumerate}
		\item \label{item: cohomology of O K is EIP}
			If $G$ is a finite flat group scheme or a smooth group scheme over $\Order_{K}$,
			then $R^{n} \pi_{\Order_{K}, \ast} G$ is P-acyclic and in $\mathcal{E}_{k}$ for any $n$
			(in particular, $R \pi_{\Order_{K}, \ast} G \in \genby{\mathcal{E}_{k}}_{D(k^{\perar}_{\et})}$).
			The object $R \artin{\pi}_{\Order_{K}, \ast} G$ is $\artin{h}$-acyclic
			with cohomologies in $\mathcal{E}_{k}$.
		\item \label{item: cohomology of K is EIP}
			If $G$ is a finite flat group scheme,
			a lattice, an abelian variety or a torus over $K$,
			then $R^{n} \pi_{K, \ast} G$ is P-acyclic and in $\mathcal{E}_{k}$ for any $n$
			(in particular, $R \pi_{K, \ast} G \in \genby{\mathcal{E}_{k}}_{D(k^{\perar}_{\et})}$).
			The object $R \artin{\pi}_{K, \ast} G$ is $\artin{h}$-acyclic
			with cohomologies in $\mathcal{E}_{k}$.
	\end{enumerate}
\end{Prop}

\begin{proof}
	These follow from Proposition \ref{prop: P acyclic implies derived h acyclic}
	and \cite[Proposition (3.4.2), (3.4.3)]{Suz14}
	except for a finite flat group scheme $G = N$ over $K$.
	For this case, the only non-trivial part is to check that
	$R^{1} \pi_{K, \ast} N$ is in $\Ind \Pro'_{\fc} \Alg / k$.
	But this follows from the proof of \cite[Proposition (3.4.3) (b)]{Suz14}.
\end{proof}

The above two propositions give some information about
the structure of $R \artin{\alg{\Gamma}}(\Order_{K}, G)$,
$R \artin{\alg{\Gamma}}(K, G)$ and
$R \artin{\alg{\Gamma}}_{x}(\Order_{K}, G)$.
For more detailed information,
see \cite[Propositions (3.4.1), (3.4.2), (3.4.3), (3.4.6); Section 9]{Suz14},
\cite[Proposition 2.5.3]{Suz18a},
\cite[Proposition 6.2 and its proof]{Suz18b}.

We can compare the cup product morphisms for
$R \alg{\Gamma}$ and $R \artin{\alg{\Gamma}}$:

\begin{Prop} \label{prop: comparison of cup product of local field} \mbox{}
	\begin{enumerate}
		\item \label{item: cup comparison for K}
			Let $\varphi \colon G \tensor^{L} G' \to G''$ be a morphism in $D(K_{\fppf})$
			such that all of $G, G', G''$ satisfy the assumption of
			Proposition \ref{prop: compatibility of cohomology functors}
			(\ref{item: comparison of cohomology for K}).
			Consider the composite of the morphisms
				\[
						R \alg{\Gamma}(K, G) \tensor^{L} R \alg{\Gamma}(K, G')
					\to
						R \alg{\Gamma}(K, G \tensor^{L} G')
					\overset{\varphi}{\to}
						R \alg{\Gamma}(K, G''),
				\]
			where the first morphism is given by \cite[(2.5.6)]{Suz18a}
			(translated into a morphism involving $\tensor^{L}$
			by the same method as the proof of
			Proposition \ref{prop: sheafified functoriality and cup product}).
			Also consider the composite of the morphisms
				\[
						R \artin{\alg{\Gamma}}(K, G) \tensor^{L} R \artin{\alg{\Gamma}}(K, G')
					\to
						R \artin{\alg{\Gamma}}(K, G \tensor G')
					\overset{\varphi}{\to}
						R \artin{\alg{\Gamma}}(K, G''),
				\]
			where the first morphism is given by
			Proposition \ref{prop: cup product for local fields}
			(\ref{item: cup product for K}).
			These composite morphisms are compatible under the isomorphism in
			Proposition \ref{prop: compatibility of cohomology functors}
			(\ref{item: comparison of cohomology for K}).
		\item \label{item: cup comparison for O K}
			Let $\varphi \colon G \tensor^{L} G' \to G''$ be a morphism in $D(\Order_{K, \fppf})$
			such that all of $G, G', G''$ satisfy the assumption of
			Proposition \ref{prop: compatibility of cohomology functors}
			(\ref{item: comparison of cohomology with support for O K}).
			Consider the composite of the morphisms
				\[
							R \alg{\Gamma}(\Order_{K}, G)
						\tensor^{L}
							R \alg{\Gamma}_{x}(\Order_{K}, G')
					\to
						R \alg{\Gamma}_{x}(\Order_{K}, G \tensor^{L} G')
					\overset{\varphi}{\to}
						R \alg{\Gamma}_{x}(\Order_{K}, G''),
				\]
			where the first morphism is given by \cite[(2.5.4)]{Suz18a}
			(translated as above).
			Also consider the composite of the morphisms
				\[
							R \artin{\alg{\Gamma}}(\Order_{K}, G)
						\tensor^{L}
							R \artin{\alg{\Gamma}}_{x}(\Order_{K}, G')
					\to
						R \artin{\alg{\Gamma}}_{x}(\Order_{K}, G \tensor G')
					\overset{\varphi}{\to}
						R \artin{\alg{\Gamma}}_{x}(\Order_{K}, G''),
				\]
			where the first morphism is given by
			Proposition \ref{prop: cup product for local fields}
			(\ref{item: cup product for O K}).
			These composite morphisms are compatible under the isomorphism in
			Proposition \ref{prop: compatibility of cohomology functors}
			(\ref{item: comparison of cohomology for O K})
			and (\ref{item: comparison of cohomology with support for O K}).
	\end{enumerate}
\end{Prop}

\begin{proof}
	(\ref{item: cup comparison for K})
	The composite morphism
		\[
					\alpha_{\ast} R \pi_{K, \ast} G
				\tensor^{L}
					\alpha_{\ast} R \pi_{K, \ast} G'
			\to
				\alpha_{\ast} \bigl(
					R \pi_{K, \ast} G \tensor^{L} R \pi_{K, \ast} G'
				\bigr)
			\to
				\alpha_{\ast} R \pi_{K, \ast}(G \tensor^{L} G')
		\]
	and the morphism
		\[
					R \artin{\pi}_{K, \ast} G
				\tensor^{L}
					R \artin{\pi}_{K, \ast} G'
			\to
				R \artin{\pi}_{K, \ast}(G \tensor^{L} G')
		\]
	are compatible by Proposition \ref{prop: comparison of structure morphisms}.
	The rest follows from Proposition \ref{prop: compatibility the associated constructions}.
	
	(\ref{item: cup comparison for O K})
	This can be proven similarly.
\end{proof}

%%%%%%%%%%%%%%%%%%%%%%%%%%%%%%%%%%%%%%%%%%%%%%%%%%%%%%%%%%%%%%%%%%%%%%%%%%%%%

\section{A duality statement in the new formulation}

As in \cite[Section 2.4]{Suz14},
we define the Serre dual functor as follows.

\begin{Def}
	Define $(\var)^{\SDual} = R \sheafhom_{k^{\ind\rat}_{\pro\et}}(\var, \Z)$.
\end{Def}

See \cite[Section 2.4, Footnote 4]{Suz14} for why it is called the Serre dual.

Now we state the duality for abelian varieties over $K$.
In \cite{Suz14}, this duality is stated using $R \alg{\Gamma}$
(i.e.\ using the functor $k' \in k^{\ind\rat} \mapsto H^{n}(\alg{K}(k'), \var)$).
Here we state it using $R \artin{\alg{\Gamma}}$
(i.e.\ using the functor $k' \in k^{\perar} \mapsto H^{n}(\alg{K}(k'), \var)$).
We here deduce the statement from the result in \cite{Suz14}
using the comparison statements in the previous section.
But note that we have developed a duality formalism with $\Spec k^{\perar}_{\et}$ in this paper well enough
so that a direct, simpler proof (using only $R \artin{\alg{\Gamma}}$) is possible.

\begin{Thm} \label{thm: local duality} \mbox{}
	Let $A$ and $B$ be abelian varieties dual to each other over $K$.
	Let $\mathcal{A}$ and $\mathcal{B}$ be their N\'eron models over $\Order_{K}$
	and $\mathcal{B}^{0}$ the open subgroup scheme of $\mathcal{B}$ with connected fibers.
	Consider the morphisms $A \tensor^{L} B \to \Gm[1]$ in $D(K_{\fppf})$
	and $\mathcal{A} \tensor^{L} \mathcal{B}^{0} \to \Gm[1]$ in $D(\Order_{K, \fppf})$
	given by the Poincar\'e bi-extension and its canonical extension to $\Order_{K}$
	(\cite[IX, 1.4.3]{Gro72}).
	Consider the morphism
		\begin{gather*}
					R \artin{\alg{\Gamma}}(K, A)
				\otimes^{L}
					R \artin{\alg{\Gamma}}(K, B)
				\to
					R \artin{\alg{\Gamma}}(K, \Gm)[1]
				\overset{\mathrm{trace}}{\to}
					\Z[1],
			\\
					R \artin{\alg{\Gamma}}(\Order_{K}, \mathcal{A})
				\otimes^{L}
					R \artin{\alg{\Gamma}}_{x}(\Order_{K}, \mathcal{B}^{0})
				\to
					R \artin{\alg{\Gamma}}_{x}(\Order_{K}, \Gm)[1]
				\overset{\mathrm{trace}}{\to}
					\Z
			\\
					R \artin{\alg{\Gamma}}_{x}(\Order_{K}, \mathcal{A})
				\otimes^{L}
					R \artin{\alg{\Gamma}}(\Order_{K}, \mathcal{B}^{0})
				\to
					R \artin{\alg{\Gamma}}_{x}(\Order_{K}, \Gm)[1]
				\overset{\mathrm{trace}}{\to}
					\Z
		\end{gather*}
	induced by Propositions \ref{prop: cup product for local fields}
	and \ref{prop: trace morphism}.
	\begin{enumerate}
		\item \label{item: duality isomorphisms}
			The resulting five morphisms
				\begin{gather*}
							R \artin{\alg{\Gamma}}(K, B)^{\SDual \SDual}
						\to
							R \artin{\alg{\Gamma}}(K, A)^{\SDual}[1],
					\\
							R \artin{\alg{\Gamma}}_{x}(\Order_{K}, \mathcal{B}^{0})
						\to
							R \artin{\alg{\Gamma}}_{x}(\Order_{K}, \mathcal{B}^{0})^{\SDual \SDual}
						\to
							R \artin{\alg{\Gamma}}(\Order_{K}, \mathcal{A})^{\SDual}
					\\
							R \artin{\alg{\Gamma}}_{x}(\Order_{K}, \mathcal{A})
						\to
							R \artin{\alg{\Gamma}}_{x}(\Order_{K}, \mathcal{A})^{\SDual \SDual}
						\to
							R \artin{\alg{\Gamma}}(\Order_{K}, \mathcal{B}^{0})^{\SDual}
				\end{gather*}
			are all isomorphisms.
		\item \label{item: duality and localization triangles}
			They form an isomorphism of distinguished triangles
				\[
					\begin{CD}
							R \artin{\alg{\Gamma}}(\Order_{K}, \mathcal{B}^{0})^{\SDual \SDual}
						@>>>
							R \artin{\alg{\Gamma}}(K, B)^{\SDual \SDual}
						@>>>
							R \artin{\alg{\Gamma}}_{x}(\Order_{K}, \mathcal{B}^{0})[1]
						\\
						@VV \wr V
						@VV \wr V
						@VV \wr V
						\\
							R \artin{\alg{\Gamma}}_{x}(\Order_{K}, \mathcal{A})^{\SDual}
						@>>>
							R \artin{\alg{\Gamma}}(K, A)^{\SDual}[1]
						@>>>
							R \artin{\alg{\Gamma}}(\Order_{K}, \mathcal{A})^{\SDual}[1]
					\end{CD}
				\]
			between the localization triangles \eqref{eq: localization triangle}.
	\end{enumerate}
\end{Thm}

\begin{proof}
	(\ref{item: duality isomorphisms})
	The morphisms
		\begin{gather*}
					R \artin{\alg{\Gamma}}(K, A) \otimes^{L} R \artin{\alg{\Gamma}}(K, B)
				\to
					R \artin{\alg{\Gamma}}(K, \Gm),
			\\
					R \alg{\Gamma}(K, A) \otimes^{L} R \alg{\Gamma}(K, B)
				\to
					R \alg{\Gamma}(K, \Gm)
		\end{gather*}
	are compatible under the isomorphisms of the terms
	by Propositions \ref{prop: groups of interest satisfy the required conditions}
	(\ref{item: cohomology of K is EIP})
	and \ref{prop: comparison of cup product of local field}
	(\ref{item: cup comparison for K}).
	The trace morphism $R \artin{\alg{\Gamma}}(K, \Gm) \to \Z$
	in Proposition \ref{prop: trace morphism}
	and the trace morphism $R \alg{\Gamma}(K, \Gm) \to \Z$
	in \cite[Proposition 2.4.4]{Suz13} are compatible
	since they both are the valuation morphism $\alg{K}^{\times} \onto \Z$.
	Hence the morphisms
		\[
					R \artin{\alg{\Gamma}}(K, B)^{\SDual \SDual}
				\to
					R \artin{\alg{\Gamma}}(K, A)^{\SDual}[1],
			\quad
					R \alg{\Gamma}(K, B)^{\SDual \SDual}
				\to
					R \alg{\Gamma}(K, A)^{\SDual}[1]
		\]
	are compatible.
	The latter is an isomorphism by \cite[Theorem (4.1.2)]{Suz14}.
	Therefore so is the former.
	The statements for $R \artin{\alg{\Gamma}}(\Order_{K}, \var)$,
	$R \artin{\alg{\Gamma}}_{x}(\Order_{K}, \var)$
	can be similarly proven.
	
	(\ref{item: duality and localization triangles})
	The stated diagram can be identified with the isomorphism of distinguished triangles
		\[
			\begin{CD}
					R \alg{\Gamma}(\Order_{K}, \mathcal{B}^{0})^{\SDual \SDual}
				@>>>
					R \alg{\Gamma}(K, B)^{\SDual \SDual}
				@>>>
					R \alg{\Gamma}_{x}(\Order_{K}, \mathcal{B}^{0})[1]
				\\
				@VV \wr V
				@VV \wr V
				@VV \wr V
				\\
					R \alg{\Gamma}_{x}(\Order_{K}, \mathcal{A})^{\SDual}
				@>>>
					R \alg{\Gamma}(K, A)^{\SDual}[1]
				@>>>
					R \alg{\Gamma}(\Order_{K}, \mathcal{A})^{\SDual}[1]
			\end{CD}
		\]
	of \cite[Proposition 2.5.4]{Suz18a}
	by Proposition \ref{prop: compatibility of cohomology functors}.
\end{proof}

The diagram in (\ref{item: duality and localization triangles}) can also be obtained directly
without referring to \cite{Suz14}.
By Proposition \ref{prop: cup product and localization},
we obtain a morphism of distinguished triangles
	\[
		\begin{CD}
				R \artin{\alg{\Gamma}}_{\Order_{K}} [\mathcal{A}, \Gm]_{\Order_{K}}[1]
			@>>>
				R \artin{\alg{\Gamma}}_{K} [A, \Gm]_{K}[1]
			@>>>
				R \artin{\alg{\Gamma}}_{x} [\mathcal{A}, \Gm]_{\Order_{K}}[2]
			\\
			@VVV
			@VVV
			@VVV
			\\
				[R \artin{\alg{\Gamma}}_{x} \mathcal{A},
				R \artin{\alg{\Gamma}}_{x} \Gm]_{k}[1]
			@>>>
				[R \artin{\alg{\Gamma}}_{K} A,
				R \artin{\alg{\Gamma}}_{x} \Gm]_{k}[2]
			@>>>
				[R \artin{\alg{\Gamma}}_{\Order_{K}} \mathcal{A},
				R \artin{\alg{\Gamma}}_{x} \Gm]_{k}[2].
		\end{CD}
	\]
Applying the morphisms $\mathcal{B}^{0} \to [\mathcal{A}, \Gm]_{\Order_{K}}[1]$ and
$B \to [A, \Gm]_{K}[1]$ to the upper triangle
and the trace isomorphism $R \artin{\alg{\Gamma}}_{x} \Gm \isomto \Z[-1]$ to the lower triangle,
we obtain a morphism of distinguished triangles
	\[
		\begin{CD}
				R \artin{\alg{\Gamma}}(\Order_{K}, \mathcal{B}^{0})
			@>>>
				R \artin{\alg{\Gamma}}(K, B)
			@>>>
				R \artin{\alg{\Gamma}}_{x}(\Order_{K}, \mathcal{B}^{0})[1]
			\\
			@VVV
			@VVV
			@VVV
			\\
				R \artin{\alg{\Gamma}}_{x}(\Order_{K}, \mathcal{A})^{\SDual}
			@>>>
				R \artin{\alg{\Gamma}}(K, A)^{\SDual}[1]
			@>>>
				R \artin{\alg{\Gamma}}(\Order_{K}, \mathcal{A})^{\SDual}[1].
		\end{CD}
	\]
Applying $\SDual \SDual$ and using (\ref{item: duality isomorphisms}),
we obtain the desired diagram.

%%%%%%%%%%%%%%%%%%%%%%%%%%%%%%%%%%%%%%%%%%%%%%%%%%%%%%%%%%%%%%%%%%%%%%%%%%%%%

\end{document}